\documentclass[a4paper,11pt,oneside,reqno]{amsart}

\usepackage{amsmath}
\usepackage{amssymb,amsmath,amsfonts,amsthm,amscd,amstext,mathtools}

\usepackage[english]{babel}
\usepackage[utf8x]{inputenc}
\usepackage[T1]{fontenc}
\usepackage[dvipsnames]{xcolor}            
\usepackage{xspace}
\usepackage[pdftex]{graphicx}
\usepackage{epstopdf}
\usepackage{mathrsfs}
\usepackage{stmaryrd}
\usepackage{enumitem}	
\usepackage{verbatim}
\usepackage{amssymb}	

\usepackage[pdftex]{hyperref}

\makeatletter
\newcommand{\leqnos}{\tagsleft@true\let\veqno\@@leqno}
\newcommand{\reqnos}{\tagsleft@false\let\veqno\@@eqno}
\reqnos
\makeatother

\usepackage{caption}           
\usepackage{url}

\usepackage[a4paper,scale={0.72,0.74},marginratio={1:1},footskip=7mm,headsep=10mm]{geometry}

\usepackage{hyperref}
\hypersetup{					
	breaklinks=true,			
	colorlinks,
    citecolor=blue,
    filecolor=black,
    linkcolor=blue,
    urlcolor=black
}

\newcommand{\ind}{{\sf 1}}

\newcommand{\Pb}{\mathbf{P}}

\newcommand{\bP}{\mathbf{P}}

\newcommand{\bE}{\mathbf{E}}

\newcommand{\Pbb}{\Pb_{\gb}}			
\newcommand{\Ebb}{\bE_{\gb}}			
\newcommand{\PbLb}{\Pb_{L,\gb}}			

\newcommand{\R}{\mathbb{R}}

\newcommand{\N}{\mathbb{N}}
\newcommand{\bbZ}{\mathbb{Z}}

\newcommand{\be}{\begin{equation}}
\newcommand{\ee}{\end{equation}}

\newcommand{\Pbbzero}{\Pb_{\gb,0}}			
\newcommand{\Ebbzero}{\bE_{\gb,0}}

\newcommand{\cH}{{\ensuremath{\mathcal H}} }
\newcommand{\cC}{{\ensuremath{\mathcal C}} }

\newcommand{\cL}{{\ensuremath{\mathcal L}} }
\newcommand{\cJ}{{\ensuremath{\mathcal J}} }
\newcommand{\cT}{{\ensuremath{\mathcal T}} }
\newcommand{\cD}{{\ensuremath{\mathcal D}} }

\newcommand{\cV}{{\ensuremath{\mathcal V}} }
\newcommand{\cB}{{\ensuremath{\mathcal B}} }

\newcommand{\cG}{{\ensuremath{\mathcal G}} }

\newcommand{\cK}{{\ensuremath{\mathcal K}} }

\renewcommand{\epsilon}{\varepsilon}
\renewcommand{\phi}{\varphi}

\newcommand{\gb}{\beta}

\newcommand{\gep}{\varepsilon}  

\newcommand{\gO}{\Omega}
\newcommand{\gl}{\lambda}
\newcommand{\gL}{\Lambda}

\newcommand{\bh}{\textbf{\textit{h}}}

\newcommand{\cst}{(\text{const.})}

\newcounter{cst}[section]		
\newcounter{svf}[section]		

\newtheorem{theorem}{Theorem}[section]

\newtheorem{proposition}[theorem]{Proposition}
\newtheorem{corollary}[theorem]{Corollary}
\newtheorem{lemma}[theorem]{Lemma}
\newtheorem{claim}[theorem]{Claim}

\theoremstyle{definition}

\newtheorem{remark}[theorem]{Remark}

\numberwithin{equation}{section}			

\newcommand{\dd}{\mathrm{d}}		

\renewcommand{\hat}{\widehat}
\renewcommand{\tilde}{\widetilde}

\DeclareMathOperator*{\argmax}{arg\,max}

\newcommand{\Imax}{I_{\mathrm{max}}}

\newcommand{\tpsi}{\tilde \psi}

\title[Asymptotics of the partition function of the collapsed IPDSAW]{A sharp asymptotics of the partition function for the collapsed  interacting partially directed self-avoiding walk }

\author{Alexandre Legrand}
\address{Universit\'e de Nantes, Laboratoire Jean Leray, 2, rue de la Houssini\`ere,44322 Nantes cedex 3, France}
\email{alexandre.legrand@univ-nantes.fr}

\author{Nicolas P\'etr\'elis}

\address{Universit\'e de Nantes, Laboratoire Jean Leray, 2, rue de la Houssini\`ere,44322 Nantes cedex 3, France}

\email{nicolas.petrelis@univ-nantes.fr}



\begin{document}

\begin{abstract}
In the present paper, we investigate the collapsed phase of the interacting partially-directed self-avoiding walk (IPDSAW) that was introduced in Zwanzig and Lauritzen (1968).
We provide sharp asymptotics of the partition function inside the collapsed phase, proving rigorously a conjecture formulated in Guttmann (2015) and Owczarek et al. (1993).
As a by-product of our result, we obtain that, inside the collapsed phase, a typical IPDSAW trajectory is made of a unique macroscopic bead, consisting of a concatenation of long vertical stretches of alternating signs, outside which 
only finitely many monomers are lying. 
    
\end{abstract}

\maketitle

\let\thefootnote\relax\footnote{{\it Acknowledgements.}  The authors thanks the Centre Henri Lebesgue ANR-11-LABX-0020-01 for creating an attractive mathematical environment.}


\section*{Notations}
Let $(a_L)_{L\leq 1}$ and $(b_L)_{L\leq 1}$ be two sequences of positive numbers. We will write 
$$a_L\underset{L\to \infty}{\sim}b_L \quad \text{if} \quad \lim_{L\to \infty} a_L/b_L=1.$$

We will also write $\cst$  to denote generic positive constant whose value may change from line to line.

\section{Introduction}

Identifying the behavior of the partition function of a lattice polymer model is in general a challenging question that 
sparked interest in both the physical and mathematical literature. In a recent survey \cite{Gut15}, some polymer models are reviewed 
that have in common that their partition functions is of the form 
\begin{equation}\label{eq:fgene}
Q_L\underset{L\to \infty}{\sim}B \mu^L \mu_1^{L^\sigma} L^g
\end{equation}
 where $B, \mu, \mu_1, \sigma, g$ are real constants depending on the coupling parameters of the model.
Among these model, the Interacting Partially Directed Self-Avoiding Walk (referred to under the acronym IPDSAW) which accounts for 
an homopolymer dipped in a poor (i.e., repulsive) solvent is conjectured to satisfy \eqref{eq:fgene} inside its collapsed phase. 
To be more specific,  based on numerics displayed in \cite{Gut15} and in \cite{OPB93} (with simulations
up to size $L=6000$), the values of $\sigma$ and $g$ are conjectured to be $1/2$ and $-3/4$. In an earlier paper \cite[Theorem 2.1]{CarPet16B}, analytic expressions where displayed for $\mu$ and $\mu_1$
while $\sigma$ was proven to be $1/2$. In the present paper, we give a full proof of \eqref{eq:fgene} for the IPDSAW in its collapsed phase, in particular we prove that $g=-3/4$ and give an analytic expression
of $B$.

\subsection{Model}
The IPDSAW was initially introduced in \cite{ZL68}. The spatial configurations of the polymer   are modeled by the trajectories of a \emph{self-avoiding} random walk on $\mathbb{Z}^2$ that only takes unitary steps \emph{upwards, downwards and to the right} (see Fig.~\ref{fig:ipdsaw}). 
To take into account the monomer-solvent interactions, one considers that, when dipped in a poor solvent, the monomers try to exclude the solvent and therefore attract one another. For this reason, any non-consecutive vertices of the walk though adjacent on the lattice are called \textit{self-touchings} (see Fig.~\ref{fig:ipdsaw}) and the interactions between monomers are taken into account by assigning an energetic reward $\beta\geq0$ to the polymer for each self-touching.

\begin{figure}
\includegraphics[scale=1.3]{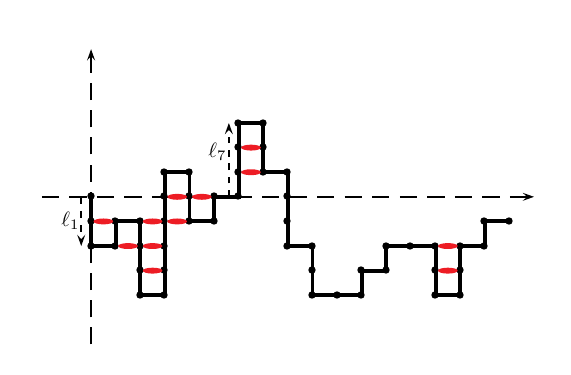}
\caption{\small Representation of an IPDSAW trajectory with length $L=48$ and horizontal extension $N=17$. Its self-touchings are shown in red.}
\label{fig:ipdsaw}
\end{figure}

It is convenient to represent the configurations of the model as families of oriented vertical stretches separated by horizontal steps. To be more specific,  for a polymer made of $L\in \N$ monomers, the set of allowed 
path is $\Omega_L:=\bigcup_{N=1}^L\mathcal{L}_{N,L}$, where $\mathcal{L}_{N,L}$ 
consists of all families made of  $N$ vertical stretches that have a total length $L-N$, that is
\begin{equation}\label{defLL}
\textstyle\mathcal{L}_{N,L}=\Bigl\{\ell:=(\ell_i)_{i=1}^N \in\mathbb{Z}^N:\sum_{n=1}^N|\ell_n|+N=L\Bigr\}.
\end{equation}
With this representation,  the modulus of a given stretch
corresponds to the number of monomers constituting this stretch  (and
the sign gives the direction upwards or downwards). For convenience, we require every configuration to end with an horizontal step, and we note that any two consecutive vertical
stretches are separated by a step placed horizontally. The latter explains why $\sum_{n=1}^N |\ell_n|$ must equal $L-N$ in order for $\ell=(\ell_i)_{i=1}^N$ to be associated with a polymer made of $L$ monomers (see Fig.~\ref{fig:ipdsaw}).

As mentioned above, the attraction between monomers 
is taken into account in the Hamiltonian associated with each path $\ell\in \Omega_L$,  by rewarding 
energetically those pairs of consecutive stretches with opposite directions, i.e.,
\begin{equation}\label{defH}
\textstyle H_{L,\beta}(\ell_1,\ldots,\ell_N)=\beta\sum_{n=1}^{N-1}(\ell_n\;\tilde{\wedge}\;\ell_{n+1}),
\end{equation}
where
\begin{equation}
x\, \tilde\wedge\, y=\begin{dcases*}
	|x|\wedge|y| & if $xy<0$,\\
  0 & otherwise.
  \end{dcases*}
\end{equation}

One can already note that large Hamiltonians will be assigned to trajectories made of few 
but long vertical stretches with alternating signs. Such paths will be referred to as collapsed configurations.
With the Hamiltonian  in hand we can define the polymer measure as
\begin{equation}\label{defpolme}
P_{L,\beta}(\ell)=\frac{e^{H_{L,\beta}(\ell)}}{Z_{L,\beta}},\quad \ell \in \Omega_L,
\end{equation}
where $Z_{L,\beta}$ is the partition function of the model, i.e.,
\begin{equation}\label{pff}
Z_{L,\beta}=\sum_{N=1}^{L}\sum_{\ell\in\mathcal{L}_{N,L}} \,  e^{H_{L,\beta}(\ell)}.
\end{equation}

\section{Main results}
The IPDSAW undergoes a collapse transition at some $\beta_c>0$ (see \cite{OPB93} or \cite[Theorem 1.3]{NGP13})
dividing $[0,\infty)$ into an extended phase $\mathcal E:=[0,\beta_c)$ and a collapsed phase $\mathcal C:=[\beta_c, \infty)$. 
Note that, when the monomer-monomer attraction is switched off ($\beta=0$), a typical configuration (sampled from $P_{L,\beta}$) has an horizontal extension $O(L)$ since, roughly speaking, every step has a positive probability (bounded from below) to be horizontal. In the extended phase, the interaction intensity $\beta$ is not yet strong enough to bring this typical horizontal extension from $O(L)$ to $o(L)$. Inside the collapsed phase, in turn, the interaction intensity is large enough to change dramatically the geometric features of a typical trajectory, which roughly looks like a compact ball with an horizontal extension $o(L)$. The asymptotics of $(Z_{L,\beta})_{L\geq 1}$ are  displayed and proven in \cite[Theorem 2.1 (1) and (2)]{CarPet16B} for the extended phase and at criticality ($\beta=\beta_c$). Inside the collapsed phase, although the exponential terms of the partition function growth rate were identified via upper and lower bounds (see \cite[Theorem~2.1 (iii)]{CarPet16B}), a full proof of such asymptotics was missing. We close this gap with Theorem~\ref{Th:4.1} below, by identifying the polynomial prefactor.  This improvement relies on a sharp local limit theorem displayed in Proposition~\ref{lemunif} for some random walk in a large deviation regime (more precisely it is constrained to be positive, cover a large area and end in~0).

Let us settle some notations that are required to state Theorem \ref{Th:4.1}.
For $\beta>0$ we let $\Pb_{\beta}$ be the following discrete Laplace probability law on $\bbZ$:
\begin{equation}\label{def:Pgb}
\Pb_\gb(\,\cdot=k)=\frac{e^{-\frac{\gb}{2}|k|}}{c_\gb}\quad,\qquad c_\gb:=\sum_{k\in\bbZ} e^{-\frac{\gb}{2}|k|} = \frac{1+e^{-\gb/2}}{1-e^{-\gb/2}},
\end{equation}
and we denote by $\cL(h)$ the logarithmic moment generating function of  $Z$ a random variable of law $\Pb_\gb$, i.e.,  
\begin{equation}\label{def:cL}
\cL(h) := \log \Ebb [e^{hZ}], \quad \quad  h\in \big(-\tfrac{\beta}{2}, \tfrac{\beta}{2}\big).
\end{equation}
Moreover, $\gb_c$ is the unique positive solution to the equation $\frac{c_\gb}{e^\gb}=1$ (see \cite[Theorem 1.3]{NGP13}).

We  define below the escape probability of a certain class of drifted random walks and we display the exponential 
decay rate of the probability that some random walk trajectories enclose an atypically large area.  
\subsubsection{Tilting and excape probability.}
 For $|h|<\beta/2$ we let  $\tilde\Pb_h$ be the probability law on $\mathbb{Z}$ defined by perturbing $\bP_\beta$ as
\begin{equation}\label{def:Pgd}
\frac{\dd \tilde\Pb_h}{\dd \Pbb} (k) = e^{h k - \cL(h)} \quad k\in \mathbb{Z}.
\end{equation}
For $h\in (0,\beta/2)$, we consider a random walk $X:=(X_i)_{i\in \N}$ such that $X_0=0$ and whose increments
$(X_{i}-X_{i-1})_{i \in \N}$ are i.i.d. with law $\tilde\Pb_h$  (i.e., with a positive drift). We denote by $\kappa(h)$  
the probability that $X$ never returns to the lower half plan, that is 
\begin{equation}\label{defKB}
\kappa(h)\,:=\,\tilde\Pb_h\big( X_i>0, \ \  \forall i\in \N) = \frac{e^{2h}-1}{e^{h+\gb/2}-1} \,,
\end{equation}
where the second identity will be proven in Lemma~\ref{convkappa}.

\subsubsection{Rate function for the area of a positive excursion.} 

We let $\cG:(-\beta,\beta)\to \R$ be defined as 
\begin{align}\label{deftiG:1}
\cG(h)&:=  \int_0^1 \cL(h(\tfrac12-x))\dd x, \quad \text{for}\quad h\in (-\beta,\beta).
\end{align}
We will prove in Lemma \ref{diffeotildeG}  that $\cG'$ is a $\cC^1$ diffeomorphism from $ (-\beta,\beta)$ to $\R$  and we let 
$q\in \R\mapsto \tilde h^q$ be its inverse function.
Considering $X:=(X_i)_{i\in \N}$ a random walk starting from the origin, whose increments
$(X_{i}-X_{i-1})_{i \in \N}$ are i.i.d. with law $\bP_{\beta}$,  we denote by $A_N(X)$ (or $A_N$ when there is no risk of confusion) the algebraic area enclosed by $X$ up to time $N$, i.e.,
\begin{equation}\label{airealg}
A_N(X):=X_1+\dots+X_{N}.
\end{equation}
\smallskip

We will prove in Proposition~\ref{lemunif} below that, for $q\in [0,\infty)\cap \frac{\N}{N^2}$, the exponential decay rate of the event 
$\{A_N=qN^2, X_N=0, X_i>0, 1\leq i\leq N-1\}$  is given by  $\tpsi:(0,\infty)\mapsto \R$ taking value
\begin{align}\label{defF}
\tpsi(q)&:= q \, \tilde h^q - \cG(\tilde h^q), \quad \text{for}\quad q\in (0,\infty).
\end{align}
At this stage, we define a positive constant that depends on $\beta$ and $q>0$ as
\begin{equation}\label{defcbq}
C_{\beta,q}:=\frac{1}{2\, \pi \  \vartheta(\tilde h^q)^{\frac12}}\   \kappa\big(\tfrac{\tilde h^q}{2}\big)^2,
\end{equation}
where $\vartheta: h\in (-\beta/2,\beta/2)\mapsto \R$ takes value 
\begin{align}\label{defvartheta}
\vartheta(h)=\int_0^1x^2 \cL''[h(x-\tfrac{1}{2})]\, dx \int_0^1 \cL''[h(x-\tfrac{1}{2})]dx-\bigg[\int_0^1x\,  \cL''[h(x-\tfrac{1}{2})]dx\bigg]^2.
\end{align}
\bigskip

%
%

\begin{theorem}\label{Th:4.1}
For any $\gb>\gb_c$, one has
\begin{equation}\label{eq:4.1}
 Z_{L,\gb}\underset{L\to \infty}{\sim}\, 
\frac{K_\beta }{L^{3/4}}\, e^{\gb L + \tilde \cG(a_\gb) \sqrt{L}},
\end{equation}
with 
\begin{align}\label{defGF}
\tilde \cG(x)&:= x \log \frac{c_\beta}{e^\beta}-x\, \tpsi(x^{-2}) , \quad x>0,
\end{align}
and  
\begin{equation}\label{defabeta}
a_\beta:=\argmax\{\tilde \cG(x)\colon x\in ]0,\infty[\},
\end{equation}
and 
\begin{equation}\label{defKdeta}
K_\beta:=\frac{2 \sqrt{2\pi} \, C_{\beta,a_\beta^{-2}}\, e^{\tpsi\,'(a_\beta^{-2})}}{\big[(1+e^{-\gb}) e^{\mathrm{arccosh}(e^{-\beta/2}\cosh(\beta))}-e^{\beta/2} (1-e^{-\beta})\big]^2\, a_\beta^2\,  \big| \tilde \cG^{''}(a_\beta)\big|^{1/2}}  .
\end{equation}


\end{theorem}
\medskip

Theorem \ref{Th:4.1} is proven in Section \ref{th41}. The proof will require to decompose a trajectory into a succession of {\it beads} 
that are sub-trajectories made of non-zero vertical stretches of alternating signs.  Inside the collapsed phase, an issue raised by physicists was to understand whether a typical trajectory contains a unique macroscopic bead or not.
Thus, for every  $\ell \in \Omega_L$  we let $N_\ell$ be its horizontal extension (i.e., $\ell\in \cL_{N_\ell,L}$) and also $|\Imax(\ell)|$  be the length of its largest bead, i.e., 
\begin{equation}\label{deflargebead}
|\Imax(\ell)|:=\max\big\{\textstyle \sum_{i=u}^v 1+|\ell_{i}| \colon\,  1\leq u\leq v\leq N_\ell,  \  \ell_i\ell_{i+1} <0 \ \ 
\forall\, u\leq i\leq v-1  \big\}.
\end{equation}
 With \cite[Thm. C]{CNP16} it is known that a typical trajectory indeed contains a unique macroscopic bead, and that at most $(\log L)^4$ 
monomers lay outside this large bead. We improve this result with the following theorem, by showing that only finitely many monomers are to be found outside the unique macroscopic bead. 
\begin{theorem}\label{thm:C}
For any $\gb>\gb_c$,
\begin{equation}\label{eq:C}
\lim_{k\to \infty} \liminf_{L\to\infty} \PbLb\big(|\Imax(\ell)|\geq L-k) =1.
\end{equation}
\end{theorem}

\medskip

\subsubsection{Outline of the paper} 

With Section \ref{Dec:prep} below, we introduce some mathematical tools of particular importance for 
the rest of the paper. Thus, in Section \ref{Renewst} we define rigorously the set containing the single-bead trajectories.
Such beads  allow us to decompose any paths in $\Omega_L$ into sub-trajectories that are not interacting with each other. Under the polymer measure, the cumulated lengths of those beads form a renewal process which will be of key importance throughout the paper. 
Section \ref{probrep} is dedicated to the random walk representation 
of the model, that was introduced initially in \cite{NGP13} and provides a probabilistic expression for
every partition function introduced in the paper.   
In Section \ref{th41} we prove Theorem  \ref{Th:4.1}  
subject to Proposition~\ref{prop:aux} 
which gives sharp asymptotics for those partition functions associated with single bead trajectories. 
Proposition~\ref{prop:aux} is proven afterwards subject to Proposition~\ref{lemunif}, 
which is the main feature of this paper since it strongly improves previous estimates (such as \cite[Prop. 2.4--2.5]{CNP16}) by providing an equivalent to 
 the probability that a random walk $X:=(X_i)_{i=0}^{N}$ of law $\bP_{\beta}$ describes a positive excursion, ends up at $0$
and encloses an area $qN^2$ for $q>0$. The proof of Proposition~\ref{lemunif} is divided into $4$ steps, displayed in Section \ref{pr:lemunif} 
after we introduce a tilted law for the random walk $X$
in such a way that the event $\{A_N(X)=q\, N^2, X_N=0\}$ becomes typical. Note that Steps 3 and 4 from Section~\ref{pr:lemunif} require local limit estimates which are displayed in Appendix~\ref{pr:estimG}, and Proposition \ref{controlderiv}
is proven in Appendix~\ref{pr:propderiv}. 
Finally, Theorem \ref{thm:C} is proven in Section \ref{unibead} by using mostly the asymptotics provided by Theorem \ref{Th:4.1}.

\section{Preparations}\label{Dec:prep}
\subsection{One bead trajectories and renewal structure}\label{Renewst}
Let $L\in \mathbb{N}$ and denote by $\Omega_L^{\,\circ}$ the subset of 
$\gO_L$ gathering those single-bead trajectories, i.e., trajectories 
made of non-zero vertical stretches that alternate orientations. Thus,  we set
$\Omega_L^{\,\circ}=\cup_{N=1}^{L/2} \cL_{N,L}^{\, \circ}$ with 
\begin{align}\label{def:gO}
\cL_{N,L}^{\, \circ}&:=\Big\{  (\ell_i)_{i=1}^{N}\in\bbZ^{N}\colon  
\sum_{i=1}^{N}|\ell_i|=L-N, \quad    \ell_i\ell_{i+1} <0 \quad \forall\,1\leq i< N\Big\}.
\end{align}
We also denote by  $ Z_{L,\beta}^{\, \circ}$ the partition function restricted to those trajectories in  $\Omega_L^{\, \circ}$, i.e., 
\begin{align}\label{def:zncirc}
Z_{L,\beta}^{\, \circ}:=\sum_{N=1}^{L/2}\,  \sum_{\ell\in  \cL_{N,L}^{\, \circ}} e^{\beta H(\ell)}.
\end{align}

In order to decompose a general trajectory into single-bead subpaths, we need to deal with the zero-length vertical stretches. This has to be done in such a way that the decomposition gives rise to a true renewal structure,
namely that the realization of a single-bead trajectory has no influence on the value of the partition function associated with the next bead. To that aim, we will integrate the zero-length stretches at the beginning of beads 
and for that we define $\hat\Omega^{\, \circ}$ the set of \emph{extended beads} as
 \begin{equation}\label{def:LLL}
\hat\Omega^{\, \circ}:=\bigcup_{L\geq 2} \hat \Omega_{L}^{\, \circ}
 \end{equation}
where  $\hat \gO_L^{\,\circ}$ is the subset of $\gO_L$  gathering those trajectories which may or may not 
 start with a sequence of zero-length stretches and form subsequently a unique bead, i.e., 
 \begin{equation}\label{def:LL}
\hat\Omega_L^{\, \circ}:=\bigcup_{k=0}^{L-2} \hat \Omega_{L}^{\, \circ,\, k}
 \end{equation}
with
 \begin{align}\label{def:extbead}
 \hat \Omega_L^{\, \circ,\, 0}&:=\Big\{  \ell\in \Omega_L^{\,\circ}\colon  \ell_1>0\Big\}\\
 \nonumber \hat \Omega_L^{\, \circ,\, k}&:=\Big\{   \ell\in \Omega_L\colon N_\ell>k, \ell_1=\dots=\ell_k=0,
 (\ell_{i+k})_{i=1}^{N_\ell-k}\in \Omega_{L-k}^{\,\circ} \Big\}, \quad k\in \{1,\dots,L-2\}.
 \end{align}
 We recall \eqref{def:zncirc} and \eqref{def:LL} and the partition function restricted to those trajectories in $\hat\Omega_L^{\,\circ}$ becomes:
\begin{equation}\label{def:ZLcirc}
\hat Z_{L,\gb}^{\, \circ}:= \sum_{k=0}^{L-2} \bigg[ \frac12 \, \ind_{\{k=0\}}\,   Z_{L,\beta}^{\, \circ} +  
\ind_{\{k\in \N\} }\,  Z_{L-k,\beta}^{\, \circ}\bigg].
\end{equation}
Note that, in the definition of  $\hat \Omega_L^{\, \circ,\, 0}$ in \eqref{def:extbead} the sign of $\ell_1$ is prescribed because
if an extended  bead does not start with a zero-length stretch, then the sign of its first stretch must be the same as that of the last stretch of the  preceding bead;hence the factor $\frac12$ in \eqref{def:ZLcirc} (see Fig.~\ref{fig:beaddecomp}). Of course this latter restriction does not apply to  the very first 
extended bead of a trajectory and this is why we define
\begin{equation}\label{def:barZLcirc}
\bar Z_{L,\gb}^{\, \circ}:= \sum_{k=0}^{L-2}  Z_{L-k,\beta}^{\, \circ}.
\end{equation}

For convenience, we define $\Omega_L^{\,c}$ the subset of $\Omega_L$ containing those trajectories ending with a non-zero stretch, i.e., 
\begin{equation}\label{def:omegalc}
\Omega_L^{\,c}:=\{\ell\in \Omega_L\colon\, \ell_{N_\ell}\neq 0\}.
\end{equation} 
At this stage we can decompose a given trajectory $\ell \in \Omega_L^{\,c}$ into extended beads 
by cutting the trajectory at times $(\tau_j)_{j=0}^{n(\ell)}$ defined as
$\tau_0=0$ and for $j\in \N$ such that $\tau_{j-1}< N_\ell$
\begin{align}\label{deftaugamma}
\tau_j&:=\max\big\{s>\tau_{j-1}\colon\, (\ell_i)_{i=1+\tau_{j-1}}^s\in \hat\Omega^{\, \circ}\quad \text{or}\quad (-\ell_i)_{i=1+\tau_{j-1}}^s\in \hat\Omega^{\, \circ} \big \}.
\end{align}
Then $n(\ell)$ is the number of beads composing $\ell$ and satisfies $\tau_{n(\ell)}=N_\ell$, thus
\begin{equation}\label{defdec}
\ell=\odot_{j=1}^{n(\ell)} \, \mathcal B_j \quad 
\text{with}\quad  
\mathcal B_j:=(\ell_{\tau_{j-1}+1},\dots,\ell_{\tau_j})\,,
\end{equation}
where $\odot$ denotes the concatenation. We also set $\mathfrak{X}_0=0$ and for $j\in \{1,\dots,n_\ell\}$, we denote by $\mathfrak{X}_{j}-\mathfrak{X}_{j-1}$ the number of steps (or monomers) 
that the $j$-th bead is made of (also 
referred to as {\it total length} of the bead), that is,
\begin{equation}\label{defle}
\mathfrak{X}_{j}-\mathfrak{X}_{j-1}=\tau_j-\tau_{j-1}+|\ell_{\tau_{j-1}+1}|+\dots+|\ell_{\tau_j}|, \quad j\in \{1,\dots,n_\ell\}.
\end{equation}
The set $\mathfrak{X}:=\{0,\mathfrak{X}_1,\dots,\mathfrak{X}_{n_\ell}\}$ contains the cumulated lengths of the beads forming $\ell$, in particular
$\mathfrak{X}_{n_\ell}=L$. 

\begin{remark}\label{beadtrajnco}
Note that  a trajectory $\ell\in \Omega_L\setminus\Omega_L^{\, c}$ may also be decomposed into extended beads as in 
(\ref{deftaugamma}--\ref{defle}). The only difference is that the very last bead is followed by a sequence of zero-length vertical stretches, i.e., 
$\mathfrak{X}_{n_\ell}=L-k$ for some $k\in \{1,\dots,L\}$ and the last $k$ vertical stretches in $\ell$ have zero-length.
\end{remark}
 
 \begin{figure}
\includegraphics[scale=0.6]{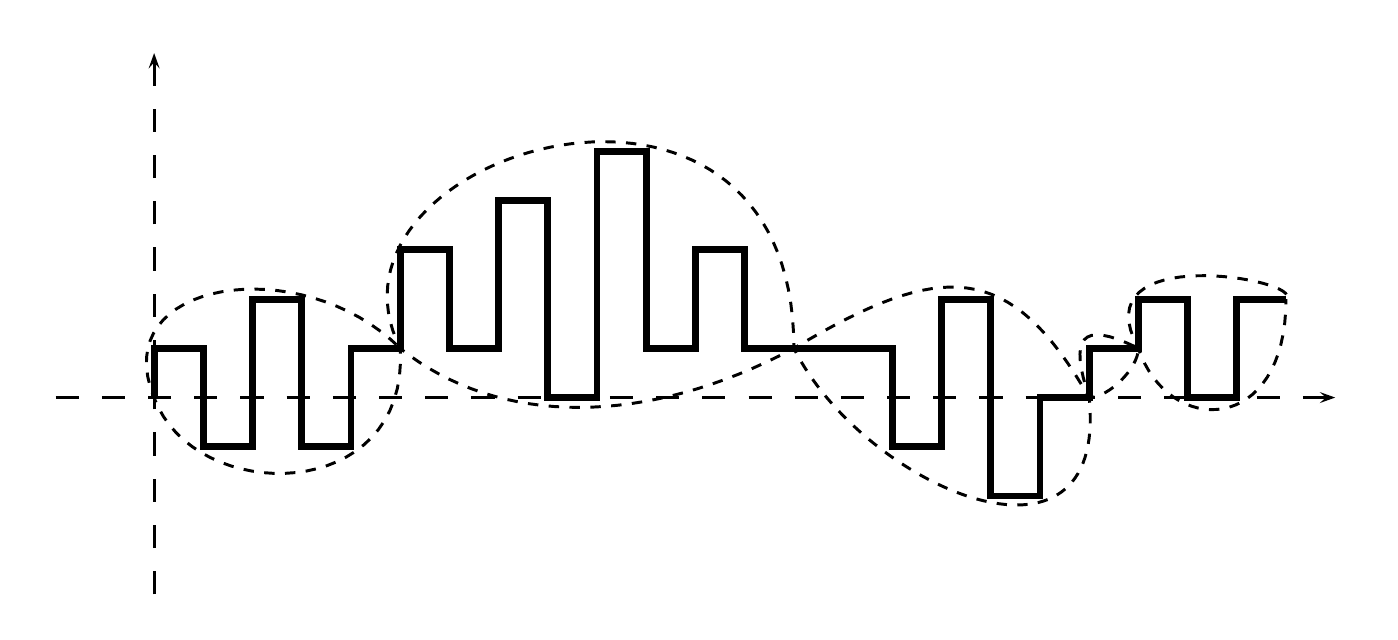}
\caption{\small Decomposition of an IPDSAW trajectory into beads. The first bead of the decomposition may start with a stretch with any sign, whereas 2nd, 4th and 5th beads start with non-zero stretches with constrained signs, and the 3rd bead starts with two zero-length stretches. Notice that the last vertical stretch of this trajectory is non-zero, so it lies in $ Z_{L,\beta}^{\, c}$.}
\label{fig:beaddecomp}
\end{figure}

By using this bead-decomposition we can rewrite $Z_{L,\beta}^{\, c}$ the partition function restricted to $\Omega_L^{\,c}$ as 
\begin{align}\label{beaddec}
\nonumber Z_{L,\beta}^{\, c}&=\sum_{r=1}^{L/2} \sum_{t_1+\dots+t_r=L}  Z_{L,\beta}^{\, c}\big(n(\ell)=r, \mathfrak{X}_i-\mathfrak{X}_{i-1}=t_i, \forall 1\leq i\leq r\big)\\
&=\sum_{r=1}^{L/2} \sum_{t_1+\dots+t_r=L}  \bar Z_{t_1,\gb}^{\, \circ}  \prod_{j=2}^r  \hat Z_{t_j,\gb}^{\, \circ},
\end{align} 
where for $D\subset \Omega_L^{\,c}$, we denote by $ Z_{L,\beta}^{\, c}(D)$ the partition function restricted to $D$ (see Fig.~\ref{fig:beaddecomp}).

\subsection{Probabilistic representation}\label{probrep}
The aim of this section is to give a probabilistic expression of the partition function $Z_{L,\beta}^{\, \circ}$ and to use it subsequently to 
provide closed expression of the generating functions associated with  $(\hat Z_{L,\beta}^{\, \circ})_{L\geq 2}$
and with $(\bar Z_{L,\beta}^{\, \circ})_{L\geq 2}$. 
\smallskip

We recall the definition of $\Pb_\gb$ in \eqref{def:Pgb} and 
for $x\in \bbZ$ we denote by $\Pb_{\beta,x}$ the law of a random walk $X:=(X_i)_{i\geq 0}$ 
starting from $x$ (i.e., $X_0=x$) and such that
$(X_{i+1}-X_i)_{i\geq 0}$ is an  i.i.d. sequence of random variables with law $\Pb_\gb$. In the case $x=0$, we will
omit the $x$-dependance of $\Pb_{\beta,x}$ when there is no risk of confusion. We also recall from \eqref{airealg} that $A_N$ defines the algebraic area 
enclosed in-between a random walk trajectory and the $x$-axis after $N$ steps.
%

Recall \eqref{def:gO} and \eqref{def:zncirc}, and let us now briefly remind the transformation that allows us to give a probabilistic representation of  $Z_{L,\beta}^{\, \circ}$ (we refer to  \cite{CNPT18} 
for a review on the recent progress made on IPDSAW by using probabilistic tools).
First, note that for 
$x,y\in\mathbb{Z}$ one can write $x\;\tilde{\wedge}\;y=\frac12\left(|x|+|y|-|x+y|\right)$. By using the latter equality to compute the Hamiltonian (recall \eqref{defH}), we  may rewrite \eqref{def:zncirc} as
\begin{align}\label{ls}
\nonumber Z_{L,\beta}^{\,\circ}
&=\sum_{N=1}^{L/2}\sum_{\substack{\ell\in\mathcal{L}_{N,L}^{\,\circ}\\ \ell_0=\ell_{N+1}=0}}\exp{\Bigl(\beta\sum_{n=1}^N{|\ell_n|}-\tfrac{\beta}{2}\sum_{n=0}^N{|\ell_n+\ell_{n+1}|}\Bigr)}\\
&=c_\beta\, e^{\beta L} \sum_{N=1}^{L/2}\left(\tfrac{c_\beta}{e^\beta}\right)^N\sum_{\substack{\ell\in\mathcal{L}_{N,L}^{\, \circ}
\\ \ell_0=\ell_{N+1}=0}}\prod_{n=0}^{N}\frac{\exp{\Bigl(-\tfrac{\beta}{2}|\ell_n+\ell_{n+1}|\Bigr)}}{c_\beta}.
\end{align}
Henceforth, for  convenience, we will assume  that any $\ell \in \cL_{N,L}^{\, \circ}$ satisfies $\ell_0=\ell_{N+1}=0$. At this stage, we denote by $B_N^{+}$ the set of those $N$-step integer-valued random walk trajectories, starting and ending at $0$ and remaining positive in-between, i.e., 
\begin{equation}\label{defbn}
B_{N}^{+}:=\big\{(x_i)_{i=0}^{N}\in \mathbb{Z}^{N+1}\colon\, X_0=X_N=0,\,  X_i>0\ \    \forall 0<i<N\big\}.
\end{equation}
It remains to notice that the application 
\begin{align}
\nonumber T_N:\{\ell\in \cL_{N,L}^{\, \circ}\colon\, \ell_1>0\}\, &\mapsto \big\{(x_i)_{i=0}^{N+1}\in B_{N+1}^{+}\colon\, A_{N}(x)=L-N\big\}\\
(\ell_i)_{i=0}^{N+1}&\to  \ ((-1)^{i-1}\,\ell_i)_{i=0}^{N+1}
\end{align}
is a one-to-one correspondence, and that for $\ell \in \cL_{N,L}^{\, \circ}$  the increments  of $T_N(\ell)$  are in modulus equal to $(|\ell_{i-1}+\ell_i|)_{i=1}^{N+1}$. Therefore, set $\Gamma_\beta:=c_\beta/e^{\beta}$
and \eqref{ls} becomes 
\begin{align}\label{ls2} 
 Z_{L,\beta}^{\, \circ} \, e^{-\beta L}
&=2\, c_\beta\,  \sum_{N=1}^{L/2}\Gamma_\beta^{\,N}\, \mathbf{P}_\beta \big(X\in B_{N+1}^+, \,  A_{N}(X)=L-N\big )
\end{align}
where the factor $2$ in the r.h.s. in \eqref{ls2} is required to take into account those $\ell\in \cL_{N,L}^{\, \circ}$ satisfying $\ell_1<0$.

With the next Lemma, we provide an expression of the generating functions   $\sum_{L\geq 2} \hat Z_{L,\beta}^{\, \circ}\,  z^L$ and 
$\sum_{L\geq 2} \bar Z_{L,\beta}^{\, \circ} \, z^L$ at $z=e^{-\beta}$. This is needed to determine $K_\beta$ in Theorem \ref{Th:4.1}. 
\begin{lemma}\label{exprgene}
For $\beta>0$, 
\begin{align}\label{exprdelta}
\nonumber \delta_1(\beta)&:= \sum_{L\geq 2} \bar Z_{L,\gb}^{\, \circ} \, \, e^{-\beta L}=\frac{2\, e^\gb}{1-e^{-\gb}}r_\gb\\
\delta_2(\beta)&:= \sum_{L\geq 2} \hat Z_{L,\gb}^{\, \circ} \, \, e^{-\beta L}=  e^\gb c_{2\gb}\,r_\gb
\end{align}
with $r_\beta:=\Ebb\big[ \ind_{\{X_1>0\}}\,  \ind_{\{X_\rho=0\}}\,  (\Gamma_\beta)^{\rho}\big]$
where $\rho:=\inf\{i\geq 1\colon\, X_i\leq 0\}$.  Moreover,
\begin{equation}
r_\beta \,=\, \left\{\begin{array}{ll}
+\infty & \quad\text{if }\beta<\beta_c\,,\\
1-e^{-\gb}-e^{-\frac\gb2+\mathrm{arccosh}\big(e^{-\beta/2} \cosh(\beta)\big)} & \quad\text{if }\beta\geq\beta_c\,.
\end{array}\right.
\end{equation}
\end{lemma}

%
\begin{proof}
We recall \eqref{ls2} and we start by computing
\begin{align}\label{ls22} 
\sum_{L\geq 2} Z_{L,\beta}^{\, \circ} \, e^{-\beta L}
\nonumber &=2\, c_\beta\,  \sum_{N\geq 1}\Gamma_\beta^{\,N}\, \sum_{L\geq 2N}  \mathbf{P}_\beta \big(X\in B_{N+1}^+, \,  A_{N}(X)=L-N\big )\\
\nonumber &=2\, c_\beta\,  \sum_{N\geq 1}\Gamma_\beta^{\,N}\,  \mathbf{P}_\beta \big(X\in B_{N+1}^+\big)\\
&=2\,  e^{\beta}\,  r_\beta.
\end{align}
Then, we use \eqref{def:ZLcirc} to write 
 \begin{align}\label{computdelta1}
\nonumber \sum_{L\geq 2} \hat Z_{L,\gb}^{\, \circ} \, \, e^{-\beta L}&:=\sum_{L\geq 2} \sum_{k=0}^{L-2} e^{-\beta L}  \bigg[ \frac12 \, \ind_{\{k=0\}}\,   Z_{L,\beta}^{\, \circ} +  
\ind_{\{k\in \N\} }\,  Z_{L-k,\beta}^{\, \circ}\bigg]\\
\nonumber &=\frac12  \sum_{L\geq 2} e^{-\beta L} Z_{L,\beta}^{\, \circ} + \sum_{k=1}^\infty e^{-\beta k} \sum_{t=2}^\infty e^{-\beta t} Z_{t,\beta}^{\, \circ} \\
 &=\Big(\frac12 +\frac{e^{-\beta}}{1-e^{-\beta}}\Big)  \sum_{L\geq 2} e^{-\beta L} Z_{L,\beta}^{\, \circ},
\end{align}
and it remains to combine \eqref{ls22} and \eqref{computdelta1} to obtain the second equality in \eqref{exprdelta}.
Using \eqref{def:barZLcirc}, the very same computation allows us to obtain the first equality in \eqref{exprdelta}.
%

Let us now compute $\Ebb[(\Gamma_\gb)^\rho]$, which will yield a formula for $r_\gb$. Indeed, the fact that the increments of $X$ follow a discrete Laplace law entails that 
$(\rho, X_1,\dots,X_{\rho-1})$ and $X_{\rho}$ are independent, and moreover that $-X_\rho$ follows a geometric law on $\N\cup\{0\}$ with parameter $1-e^{-\beta/2}$. Thereby,
\begin{align}\label{eq:rho1}
\nonumber \Ebb\big[(\Gamma_\gb)^\rho\big]\,&=\, \Ebb\big[1_{\{X_1>0\}}(\Gamma_\gb)^\rho\big] + \Gamma_\gb \,\Pbb(X_1\leq 0)\\ 
\nonumber &=\,\frac{r_\gb}{\Pbb(-X_\rho=0)} + \frac{c_\gb}{e^\gb}\left(\frac1{c_\gb}+\frac12\Big(1-\frac1{c_\gb}\Big)\right)\\
&=\, \frac{r_\gb+e^{-\gb}}{1-e^{-\gb/2}}\,,
\end{align} 
which will conclude the proof.

It is a  straightforward application of \cite[(4.5)]{CC13}
that there exists $c>0$ depending on $\beta$ only such that 
\begin{equation}\label{equivexc}
\Pbb(\rho=t)\,\sim\, \frac{c}{t^{3/2}}\qquad\text{as}\quad t\to \infty\;.
\end{equation}  
Recall that $\gb\mapsto\Gamma_\gb$ is decreasing on $(0,\infty)$; in particular when $\gb<\gb_c$, one has $\Gamma_\gb>1$, so \eqref{eq:rho1} and \eqref{equivexc} imply $r_\gb=+\infty$.

When $\gb\geq\gb_c$, notice that $(e^{-\zeta X_n+\log(\Gamma_\gb)n})_{n\geq0}$ is a martingale under $\Pbbzero$ if and only if $\cL(-\zeta )=-\log(\Gamma_\gb)$. Recalling $\Gamma_\gb\leq1$ and that $\cL$ is decreasing, not bounded on $(-\gb/2,0]$, there is a unique $\zeta_\gb\in[0,\gb/2)$ satisfying this equality, given by
\begin{equation}\label{eq:rho2}
\zeta_\gb \,=\, \mathrm{arccosh}\big((1-\Gamma_\gb)\cosh(\gb/2) + \Gamma_\gb\big)\,=\mathrm{arccosh}\big(e^{-\beta/2} \cosh(\beta)\big).
\end{equation}
Thanks to a stopping time argument (we do not write the details here), we finally obtain
\[
\Ebb\big[(\Gamma_\gb)^\rho\big] \,=\, \Ebb\Big[e^{-\zeta_\gb X_\rho}\Big]^{-1} \,=\, \frac{1-e^{\zeta_\gb-\gb/2}}{1-e^{-\gb/2}}\,,
\]
which concludes the proof by recollecting \eqref{eq:rho1} and \eqref{eq:rho2}.
\end{proof}
To end this section we state and prove the following corollary that will be needed in the proof of Theorem \ref{Th:4.1} (see Section \ref{th41}) below.
\begin{corollary}\label{controlbeta2}
For any $\beta>\beta_c$, we have $\delta_2(\beta)<1$.
\end{corollary}
\begin{proof}
We observe that $\beta\to \delta_2(\beta)$ is decreasing simply
because for every $L\geq 2$ and every $\ell \in \Omega_L$ the quantity $\beta\to  H_{L,\beta}(\ell)-\beta L$ is non-increasing and even decreasing if  $ H_{L,\beta}(\ell)<\beta L$ (recall \eqref{defH}).  Therefore, the corollary will be proven once we show $\delta_2(\beta_c)=1$. Recall that $\Gamma_{\gb_c}=1$, which ensures that $e^{\gb_c/2}$ is a solution to $X^3-X^2-X-1=0$ and that $\zeta_{\beta_c}=0$ in \eqref{eq:rho2}. This implies both
\[r_{\gb_c} \,=\, 1- e^{-\gb_c} - e^{-\gb_c/2} \,=\, e^{-3\gb_c/2}\,,\]
and
\[\delta_2(\gb_c) \,=\, e^{-\beta_c/2} \bigg(\frac{1+e^{-\beta_c}}{1-e^{-\beta_c}}\bigg) \,=\, 1\,, \]
which concludes the proof.
\end{proof}

\section{Proof of Theorem \ref{Th:4.1}}\label{th41}

In this section we provide a sharp estimate for the partition function of single-bead trajectories in Proposition~\ref{prop:aux}, with which we prove Theorem~\ref{Th:4.1}. The proof of Proposition~\ref{prop:aux} is displayed afterwards subject to Proposition~\ref{lemunif} and Lemma~\ref{rest-ext}.
\begin{proposition}\label{prop:aux}
For $\gb>\gb_c$, there exists $K^{\, \circ}_\beta>0$ such that
\begin{equation}\label{eq:4.22}
Z^{\, \circ}_{L,\gb}\underset{L\to \infty}{\sim}
\frac{K_\beta^{\, \circ} }{L^{3/4}}e^{\gb L + \tilde \cG(a_\gb) \sqrt{L}}.
\end{equation}
\end{proposition}
\begin{corollary}\label{prop:4.2}
For $\gb>\gb_c$, there exist $\hat K_\beta>0$ and $\bar K_\beta>0$ such that
\begin{equation}\label{eq:4.2}
\hat Z^{\, \circ}_{L,\gb} \underset{L\to \infty}{\sim}
\frac{\hat K_\beta }{L^{3/4}}e^{\gb L + \tilde \cG(a_\gb) \sqrt{L}}\quad \text{and}\quad \bar Z^{\, \circ}_{L,\gb}\underset{L\to \infty}{\sim}
\frac{\bar K_\beta }{L^{3/4}}e^{\gb L + \tilde \cG(a_\gb) \sqrt{L}}\,.
\end{equation}
\end{corollary}
We observe that \eqref{eq:4.22} is a substantial improvement of \cite[Prop. 4.2]{CNP16}, where the polynomial factors were $\frac{1}{L^{\kappa}}$ with $\kappa>1$ in the lower bound, and $\frac{1}{\sqrt{L}}$ in the upper bound. Moreover Corollary~\ref{prop:4.2} is a straightforward consequence of Proposition~\ref{prop:aux}. Indeed, let us define for convenience
\begin{equation}\label{def:h}
h(n):=\frac{1}{n^{3/4}}e^{\tilde \cG(a_\gb) \sqrt{n}}, \quad n\in \N\,.
\end{equation}
Recollecting \eqref{def:ZLcirc}, we write
\[\frac{e^{-\gb L}}{h(L)} \hat Z_{L,\gb}^{\, \circ} \,=\, \frac12\frac{e^{-\gb L}}{h(L)} Z_{L,\gb}^{\, \circ} + \sum_{k=1}^{L-2} e^{-\gb k} \frac{e^{-\gb (L-k)}}{h(L)} Z_{L-k,\gb}^{\, \circ}\,,\]
and similarly for \eqref{def:barZLcirc}. Noticing that $h(L)\sim h(L-k)$ as $L\to\infty$ for any $k\in\N$, \eqref{eq:4.22} and dominated convergence imply that \eqref{eq:4.2} holds true with 
\begin{equation}\label{compkb}
\hat K_\beta=K_\beta^{\, \circ}\,  \frac{1+e^{-\beta}}{2 (1-e^{-\beta})} \quad \text{and} \quad \bar K_\beta=K_\beta^{\, \circ}\,  \frac{1}{1-e^{-\beta}}.
\end{equation}

\begin{proof}[Proof of Theorem~\ref{Th:4.1} subject to Proposition~\ref{prop:aux}]
Recall the definitions of $\delta_1(\beta)$ and $\delta_2(\beta)$ in \eqref{exprdelta}, and define two probability laws $q_1$ and $q_2$ on 
$\N$ by 
\begin{align}
q_1(t)&:= \delta_1(\beta)^{-1} \, \bar Z_{t,\gb}^{\, \circ} \,  e^{-\beta t} \quad t\geq 2,\\
q_2(t)&:= \delta_2(\beta)^{-1} \,  \hat Z_{t,\gb}^{\, \circ}\,  e^{-\beta t} \quad t\geq 2.
\end{align}
For $r\in \N$, we denote by $q_2^{r\, *}$ the convolution product of $r$ times $q_2$ and by $q_1*q_2^{r*}$ the convolution product  between $q_1$ and $q_2^{r\, *}$.
This allows us to rewrite \eqref{beaddec} under the form 
\begin{align}\label{beaddec2}
\tilde  Z_{L,\beta}^{\, c}:= e^{-\beta L} Z_{L,\beta}^{\, c}&=
\frac{\delta_1(\beta)}{\delta_2(\beta)}\sum_{r\geq 1} \delta_2(\beta)^r  \, \Big[q_1*q_2^{(r-1)*}\Big](L).
\end{align}
Recalling \eqref{def:h}, Corollary~\ref{prop:4.2} implies that 
\begin{align}\label{equivq12}
\nonumber q_1(n)&\underset{n\to \infty}{\sim} u_\beta\,  h(n), \quad \text{with}\quad u_\beta:=\frac{\bar K_\beta}{\delta_1(\beta)} \\
q_2(n)&\underset{n\to \infty}{\sim} v_\beta\,  h(n), \quad \text{with}\quad v_\beta:=\frac{\hat K_\beta}{\delta_2(\beta)} 
\end{align}

At this stage, Theorem~\ref{Th:4.1} will be a straightforward consequence of Claims~\ref{subexponen} and~\ref{boundconvol} below.
Those claims are proven in \cite[Corollary 4.13 and Theorem 4.14]{FKZ11} in the case where $q_1\equiv q_2$.  However, \eqref{equivq12}
guarantees that the proof in \cite{FKZ11} can easily be adapted to our case since $q_1$ and $q_2$ enjoy the same asymptotic behavior (up to a constant).

\begin{claim}\label{subexponen}
For $\beta>0$ and $r\in\N\cup\{0\}$ it holds that $ q_1*q_2^{r*}(n)\sim_{n\to \infty} (u_\beta+r\,  v_\beta)\,  h(n)$.
\end{claim}

\begin{claim}\label{boundconvol}
For $\beta>0$ and $\gep>0$ there exists $n_0(\gep)\in \N$ and $C(\gep)>0$ such that 
\begin{equation}\label{boundunifconv}
q_1*q_2^{r*}(n)\leq C(\gep)\,  (1+\gep)^r\,  h(n), \quad n\geq n_0(\gep), r\in \N\cup\{0\}
\end{equation} 
\end{claim}
It remains to use the dominated convergence Theorem to conclude from Claims~\ref{subexponen} and~\ref{boundconvol} that for $\delta<1$
\begin{equation}\label{limitconv}
\lim_{n\to \infty} \frac{1}{h(n)} \sum_{r\geq 1} \delta^r\, \Big[q_1*q_2^{(r-1)*}\Big](n)=u_\beta \sum_{r\geq 1} \delta^r+ \delta\,  v_\beta \sum_{r\geq 1} r\,  \delta^r=\frac{u_\beta\, \delta}{1-\delta}+\, \frac{v_\beta\, \delta^2}{(1-\delta)^2}
\end{equation}
Combining \eqref{beaddec2} with \eqref{limitconv} at $\delta=\delta_2(\beta)$ (recall that $\delta_2(\beta)<1$ by Corollary \ref{controlbeta2}) we obtain that 
\begin{equation}\label{convpartfunc}
\lim_{L\to \infty} \frac{1}{h(L)} \tilde  Z_{L,\beta}^{\, c}=\frac{\bar K_\beta}{1-\delta_2(\beta)}+\frac{\hat K_\beta\,  \delta_1(\beta)}{({1-\delta_2(\beta))^2}}.
\end{equation}

To complete the proof of Theorem~\ref{Th:4.1}, it remains to take into account trajectories that end with some zero-length stretches. 
To that aim, we recall \eqref{def:omegalc} and  we partition $\Omega_L$ into subsets whose trajectories are ending with a prescribed number of zero-length stretches, i.e., 
 \begin{align}\label{trajend0}
 \Omega_L&=\cup_{k=0}^L\Big\{   \ell\in \Omega_L\colon N_\ell\geq k,  (\ell_{i})_{i=1}^{N_\ell-k}\in \Omega_{L-k}^{\,c}, \, \ell_{N_\ell-k+1}=\ell_{N_\ell-k+2}=\dots=\ell_{N_\ell}=0\Big\}.
 \end{align}
 By using this decomposition, we obtain that 
 \begin{equation}
 \tilde Z_{L,\beta}:=Z_{L,\beta} \, e^{-\beta L}=\sum_{k=0}^L e^{-\beta k} \tilde Z_{L-k,\beta}^{\,c}
 \end{equation}
 and then, using \eqref{convpartfunc} and dominated convergence we conclude that  
 \begin{equation}\label{convpartfunc2}
K_\beta:=\lim_{L\to \infty} \frac{1}{h(L)} \tilde  Z_{L,\beta}=\frac{1}{1-e^{-\beta}}\, \Big[ \frac{\bar K_\beta}{1-\delta_2(\beta)}+\frac{\hat K_\beta\,  \delta_1(\beta)}{({1-\delta_2(\beta))^2}}\Big],
\end{equation}
which completes the proof of Theorem~\ref{Th:4.1}, by combining \eqref{convpartfunc2} with \eqref{compkb} and \eqref{exprdelta}, and by writing
\be\label{valuekb}
K_\beta\,=\,K_\beta^{\, \circ} \,e^{-\gb}\, \big[(1+e^{-\beta})\,  e^{\mathrm{arccosh}(e^{-\beta/2}\cosh(\beta))}-e^{\beta/2} (1-e^{-\beta})\big]^{-2}\,,
\ee
where $K_\beta^{\, \circ}$ is computed below in \eqref{valkbo}.
\end{proof}

\begin{proof}[Proof of Proposition \ref{prop:aux}]
Let us now prove  Proposition~\ref{prop:aux} subject to Proposition~\ref{lemunif} and Lemma~\ref{rest-ext} below. 
The proof of Proposition~\ref{lemunif} is postponed to section \ref{pr:lemunif}, whereas Lemma~\ref{rest-ext} was already stated and proven in \cite[Lemma 4.4]{CNP16}
so we will not repeat the proof in the present paper. 
\smallskip

For $\beta>0$ and $q>0$ recall the definitions of $\tpsi(q)$ and $C_{\beta,q}$ from \eqref{defF} and \eqref{defcbq}.
\begin{proposition}\label{lemunif}
Let $[q_1,q_2]\subset (0,\infty)$ and $N\in \N$. We have that for $N\in \N$ such that 
$qN^2\in \N$,
\begin{equation}\label{lemunif2}
\Pb_{\gb}(\mathcal{V}_{N,\, qN^2} )\,=\, \frac{C_{\beta,q}}{N^2} \ e^{-N\tpsi(q)}  (1+o(1)) \,,
\end{equation}
uniformly in $q\in [q_1,q_2]$, with 
\begin{equation}\label{def:nu}
\mathcal{V}_{n,k}:=\{X\colon X_n=0,\,  A_n=k, X_i>0, 0<i<n\}, \quad (n,k)\in (\N_0)^2. 
\end{equation}

\begin{remark}\label{linkwithWP}
For the proof of Proposition \ref{lemunif} (see Section \ref{pr:lemunif}), we took inspiration from 
\cite{PerfWach19} where a slightly different problem is considered. To be more specific, the authors consider a random walk
$(Y_i)_{i\geq 0}$ with a negative drift and a light tail such that  the moment generating function $\phi(t):=E(e^{t\, Y_1})$
 satisfies that there exists a $\lambda>0$ such that $\phi(\lambda)=1$, $\phi^{'}(\lambda)<\infty$ and $\phi^{''}(\lambda)<\infty$.   
For such a walk, they provide the assymptotics of the joint law of $\tau:=\inf\{i\geq 1\colon \, Y_i\leq 0\}$ and of $A_{\tau-1}$
as the latter becomes large. When applied in our framework, \cite[Theorem 1 and 2]{PerfWach19} prove that 
for $p>0$ and $q>p/2$ there exist  $C_1, C_2>0$ such that 
$$\Pb_{\gb}(X_N\leq pN, A_N=q N^2, \, X_i>pi,\,  0< i<N)\underset{N\to \infty} {\sim}\frac{C_1}{N^2} e^{- C_2 N}.$$
The case $p=0$, which is the object of Proposition \ref{lemunif} is not covered by \cite{PerfWach19} though, which is why
we provide a complete proof in Section \ref{pr:lemunif}.
\end{remark}

\end{proposition}
\smallskip

\begin{lemma}\label{rest-ext}
Let $\beta>\beta_c$, there exists $(a_1,a_2)\in (0,\infty)^2$ such that 
\begin{equation}\label{resthext}
\lim_{L\to \infty} \frac{Z_{L,\beta}^{o}([a_1,a_2] \sqrt{L})}{Z_{L,\beta,\delta}^{o,+}}=1\,.
\end{equation}
\end{lemma}
\smallskip

We start the proof of  Proposition~\ref{prop:aux} by recalling \eqref{ls2} and the equality $\frac{c_\beta}{\Gamma_\beta}=e^{\beta}$, which allow us to write for $L\in \N$,
\begin{align}\label{rwrep}
 Z_{L,\beta}^{\, \circ}&=2\,  c_\beta \, e^{\beta L}\,  \tilde Z_{L,\beta}^{\, \circ}:=2 \, e^{\beta (1+L)}\, \sum_{N=2}^{1+L/2}  \Gamma(\beta)^{N}  \, \Pb_{\gb}(\cV_{N,L-N+1})\,.
\end{align}
Lemma \ref{rest-ext} guarantees us that it suffices to consider
 \begin{align}\label{defQL}
Q_{L,\gb}:&=\sum_{N=a_1 \sqrt{L}}^{a_2 \sqrt{L}} \Gamma(\beta)^N  \, \Pb_{\gb}(\cV_{N,L-N+1})\\
\nonumber &=\sum_{x\in[a_1,a_2]\cap\frac{\N}{\sqrt{L}}}  \Gamma(\beta)^{x\sqrt{L}}  \, \Pb_{\gb}\big[\cV_{x\sqrt{L},\,  q_{L}(x) (x \sqrt{L})^2}\big],
\end{align}
with $q_L(x):=\frac{L-x\sqrt{L}+1}{x^2L}$. 
We note that $x\in [a_1,a_2]$ yields that for $L$ large enough, 
$q_L(x)\in [\frac{1}{2 a_2^2},\frac{2}{a_1^2}]$ so that we can apply Proposition~\ref{lemunif} for every $x\in [a_1,a_2]\cap\frac{\N}{\sqrt{L}}$ in the r.h.s. of \eqref{defQL} and obtain 
 \begin{align}\label{QL2}
Q_{L,\gb}&\underset{L\to \infty}{\sim}\sum_{x\in[a_1,a_2]\cap\frac{\N}{\sqrt{L}}}  \frac{C_{\beta,q_L(x)}}{x^2\, L}
e^{x\sqrt{L} \, [\log \Gamma(\beta) - \tpsi(q_L(x))]}.
\end{align}
By definition (see \eqref{defF}) $\tpsi$ is  $\cC^2$ on $(0,\infty)$ and therefore,
uniformly in $x\in [a_1,a_2]$ we get
\begin{equation}\label{fq}
\tpsi(q_L(x))=\tpsi(x^{-2})-\tpsi\,'(x^{-2})\, \tfrac{1}{x\sqrt{L}}+O(\tfrac{1}{L})
\end{equation} 
such that \eqref{QL2} becomes 
 \begin{align}\label{QL3}
Q_{L,\gb}&\underset{L\to \infty}{\sim}\, \frac{1}{L} \sum_{x\in[a_1,a_2]\cap\frac{\N}{\sqrt{L}}}  \frac{C_{\beta,q_L(x)}e^{\tpsi\,'(x^{-2})}}{x^2}
e^{\sqrt{L}\,  \tilde \cG(x)},
\end{align}
with $x\in (0,\infty)\mapsto \tilde \cG(x):= x \log \Gamma(\beta)- x\, \tpsi(x^{-2})$ a function already used in \cite[(1.27)]{CNP16}, which is $\cC^2$, negative, has negative second derivative (and therefore is strictly concave on $(0,\infty)$), and reaches its unique maximum at $a_\gb\in[a_1,a_2]$.
\medskip


At this stage we pick $R>0$ and we set $\cT_{R,L}:=[a_\beta -\frac{R}{L^{1/4}},a_\beta+\frac{R}{L^{1/4}}]\cap \frac{\N}{\sqrt{L}}$, and
\begin{align}\label{deftildeQ}
\tilde Q_{L,\beta}^{R,+}:=& \sum_{x\in \cT_{R,L}}   \frac{C_{\beta,q_L(x)} e^{\tpsi\,'(x^{-2})}}{x^2}
e^{\sqrt{L}\,  \tilde \cG(x)}\,,   \\
\nonumber \tilde  Q_{L,\beta}^{R,-}:=&  \sum_{x\in [a_1,a_2]\setminus \cT_{R,L}\cap \frac{\N}{\sqrt{L}}}   \frac{C_{\beta,q_L(x)}e^{\tpsi\,'(x^{-2})}}{x^2}
e^{\sqrt{L}\,  \tilde \cG(x)}\,. 
\end{align}
We recall \eqref{defcbq} and we note that for $\beta>0$, the function $q\in (0,\infty) \mapsto C_{\beta,q}$ is continuous since $x\in (0,\beta/2)\mapsto \kappa(x)$
is continuous (see Lemma \ref{convkappa}) as well as $q\mapsto \tilde h^q$ (see \eqref{def:cLgL} and 
\eqref{def:tildebh}). Moreover $q_L(x)$ converges to $\frac{1}{a_\beta^2}$ uniformly in $x\in \cT_{R,L}$ and therefore
\begin{equation}\label{limite}
\lim_{L\to \infty}  C_{\beta,q_L(x)}=C_{\beta,a_\beta^{-2}} \quad \text{uniformly in $x\in \cT_{R,L}$}
\end{equation}
so that we can rewrite 
\begin{equation}\label{QR+}
\tilde Q_{L,\beta}^{R,+}\underset{L\to \infty}{\sim}\, \frac{ C_{\beta, a_\beta^{-2}}\, \,  e^{\tpsi\,'(a_\beta^{-2})}}{a_\beta^2}
\sum_{n=-R\,L^{1/4}}^{R\, L^{1/4}} e^{\sqrt{L}\,  \tilde \cG\big(a_\beta+ \frac{n}{\sqrt{L}}\big)}   
\end{equation}
where we have changed the summation indices for computational convenience. Note that, in the argument of $\tilde \cG$ in \eqref{QR+},  we should have considered  $\lfloor a_\beta\sqrt{L}\rfloor)/\sqrt{L}$ rather than $a_\beta$   which
may not belong to  $\N/\sqrt{L}$. However, $\tilde \cG$ being $\cC^1$ on $[a_1,a_2]$
the equivalence in \eqref{QR+} remains true.   Since $a_\beta$ is a maximum of $\tilde \cG$ that is $\cC^2$ we can write the following expansion of $\tilde \cG$,
\begin{equation}
 \tilde \cG\big(a_\beta+ \frac{n}{\sqrt{L}}\big)= \tilde \cG(a_\beta)+ \tilde \cG^{''}(a_\beta)\,   \frac{n^2}{2 L}+O\big(\frac{n^3}{L^{3/2}}\big)
\end{equation}
and use it in the last term in  \eqref{QR+} to write
\begin{align}\label{QR+2}
\sum_{n=-R\,L^{1/4}}^{R\, L^{1/4}} e^{\sqrt{L}\,  \tilde \cG(a_\beta+ \frac{n}{\sqrt{L}})}  
\nonumber &\underset{L\to \infty}{\sim}\, e^{\tilde \cG(a_\beta) \sqrt{L}} \sum_{n=-R\,L^{1/4}}^{R\, L^{1/4}} 
e^{\frac{1}{2} \tilde \cG''(a_\beta) (\frac{n}{L^{1/4}})^2}  \\
&\underset{L\to \infty}{\sim}\,  L^{\frac14} \, e^{\tilde \cG(a_\beta) \sqrt{L}} \,  \int_{-R}^{R} e^{\,  \tilde \cG''(a_\beta)\,  \frac{x^2}{2} } dx
\end{align}
where we have used a Riemann sum to obtain the last equality in \eqref{QR+2} since $\tilde \cG''(a_\beta)$
is negative. It remains to combine \eqref{QR+} with \eqref{QR+2} to obtain 
\begin{equation}\label{finQL+}
Q_{L,\beta}^{R,+}\underset{L\to \infty}{\sim}\,  \frac{C_{\beta,\, a_\beta^{-2}}\,  e^{\tpsi\,'(a_\beta^{-2})}}{a_\beta^2}\,  L^{\frac14} \, e^{\tilde \cG(a_\beta) \sqrt{L}} \,  \int_{-R}^{R} e^{\, \tilde \cG''(a_\beta) \,  \frac{x^2}{2}} dx.
 \end{equation}
 
 Let us now consider $Q_{L,\beta}^{R,-}$. To that aim, we bound it from above as 
 \begin{align}\label{boundQL-}
 \tilde  Q_{L,\beta}^{R,-}\leq & M_{\beta,L}\, \sum_{x\in [a_1,a_2]\setminus \cT_{R,L}\cap \frac{\N}{\sqrt{L}}}   e^{\sqrt{L}\,  \tilde \cG(x)}.
 \end{align}
 with
 \begin{equation}\label{defMB}
 M_{\beta,L}:=  \max_{x\in  [a_1,a_2]} \frac{C_{\beta,q_L(x)}e^{\tpsi\,'(\frac{1}{x^2})}}{x^2}.
 \end{equation}
The continuity on $(0,\infty)$ of  both $q\mapsto C_{\beta,q}$ and $\tpsi\,'$ and the fact that, for $L$ large enough,  $q_L(x)\in [1/(2a_2^2), 2/a_1^2]$ for $x\in [a_1,a_2]$,  guarantees us that there exists a $M>0$ such that $M_{\beta,L}\leq M$ for every $L\geq 1$

 Recalling that $\tilde \cG$ is $\cC^2$, strictly concave and reaches its maximum at $a_\beta$ there exists $c>0$ such that
 $\tilde \cG(x)\leq  \tilde \cG(a_\gb) - c\,  (x-a_\gb)^2$ for $x\in [a_1,a_2]$ and therefore the sum in the r.h.s. in \eqref{boundQL-} can be bounded above as 
 \begin{align}\label{boundSum}
 \sum_{x\in [a_1,a_2]\setminus \cT_{R,L}\cap \frac{\N}{\sqrt{L}}}   e^{\sqrt{L}\,  \tilde \cG(x)}&\leq 2 L^{1/4} e^{ \tilde \cG(a_\gb)  \sqrt{L}}
  \frac{1}{L^{1/4}}\sum_{n= R\, L^{1/4}}^{\infty} e^{-c\,  (\frac{n}{L^{1/4}})^2 }
 \end{align}
The function $x\mapsto e^{-c x^2}$ being non increasing on $[0,\infty)$, a standard comparison between sum and integral yields that 
for $L$ large enough and every $R\geq 2$,
 \begin{equation}\label{convRieman}
 \frac{1}{L^{1/4}}\sum_{n= R\, L^{1/4}}^{\infty} e^{-c\,  (\frac{n}{L^{1/4}})^2 }\leq \int_{R-1}^\infty e^{-c\,  x^2} dx.
 \end{equation}
 It remains to use (\ref{boundQL-}--\ref{convRieman}) to conclude that for $L$
 large enough and $R\geq 2$, 
  \begin{align}\label{boundQL-2}
 \tilde  Q_{L,\beta}^{R,-}\leq & 2\, M\,   L^{1/4} e^{ \tilde \cG(a_\gb)  \sqrt{L}} \int_{R-1}^\infty e^{-c\,  x^2} dx.
 \end{align}
 
 We recall that $\int_\R e^{\tilde \cG''(a_\beta) \frac{x^2}{2}} dx=\sqrt{-2\pi/\tilde \cG''(a_\beta)}$. Then, we combine 
 \eqref{QL3} with \eqref{deftildeQ},  \eqref{finQL+} and \eqref{boundQL-2} to claim that 
 \begin{equation}\label{eq:equiv}
 Q_{L,\beta}\underset{L\to \infty}{\sim}\frac{\sqrt{2\pi}\, C_{\beta, a_\beta^{-2}}\   e^{\tpsi\,'(a_\beta^{-2})}}{a_\beta^2\,  \big| \tilde \cG''(a_\beta)\big|^{1/2}} \frac{1}{L^{3/4}}\,   e^{\tilde \cG(a_\beta) \sqrt{L}}\,,
 \end{equation}
 and it suffices to recall \eqref{rwrep} to complete the proof of Proposition \ref{prop:aux} with 
 \begin{equation}\label{valkbo}
 K_\beta^{\, \circ}=2\, e^\beta\,   \frac{\sqrt{2\pi}\, C_{\beta,a_\beta^{-2}}\,  e^{\tpsi\,'(a_\beta^{-2})}}{a_\beta^2\  \big| \tilde \cG''(a_\beta)\big|^{1/2}}\,.
 \end{equation}
\end{proof}

\section{Proof of Proposition \ref{lemunif}}\label{pr:lemunif}

Proposition~\ref{lemunif} is a substantial improvement of \cite[Prop. 2.5]{CNP16}, since this latter proposition only allowed us to bound from below the quantity $\Pb(\cV_{n,qn^2})$ with a polynomial factor $1/n^{\gamma}$ ($\gamma>2$) instead of 
$1/n^{2}$ in the present Lemma.  
Let us first recall some results on the large deviation principle 
satisfied by the sequence of random vectors $(\frac1n A_{n-1}, X_n)_{n\geq 1}$. We then provide an outline of the proof of Proposition~\ref{lemunif}, which is divided into 4 steps, corresponding to Sections \ref{Step1}, \ref{Step2}, \ref{Step3} and \ref{Step4} respectively.

\subsection*{Change of measure}\label{changeof}
Let $X:=(X_i)_{i\in \mathbb{N}}$ be a random walk starting from the origin and whose increments are iid with law $\Pbb$
(recall \eqref{def:Pgb}). For $|h|<\beta/2$, recall the definitions of $\cL(h)$ and $\tilde\Pb_h$ in \eqref{def:cL}, and \eqref{def:Pgd}.
Recall also the definition of $A_N(X)$ in \eqref{airealg}.

Large deviations estimates for the random vector $\Lambda_n:=\big(\frac{1}{n} A_{n-1}(X),X_n\big)$ are displayed in \cite{DH96}. 
Typically, one is interested in the probability of events as
$$\{\tfrac1n \Lambda_n= (q,p)\}\quad \text{with} \quad  (q,p)\in \R^*\times\R,$$
which requires to introduce  tilted probability laws of the form
\begin{equation}\label{def:Pnh}
\frac{\dd \Pb_{n,\bh}}{\dd \Pbbzero} (X) = e^{\bh\cdot\gL_n- \cL_{\gL_n}(\bh)} \quad \text{with}\quad \cL_{\gL_n}(\bh):= \log {\bf E}_{\beta,0}[e^{\bh\cdot\gL_n}].
\end{equation}
where $\bh:=(h_0,h_1)\in \cD_{\gb,n}$
with   
$$\cD_{\gb,n}:=\; \big\{(h_0,h_1)\in\R^2\,;\;|h_1|<\gb/2,\; |(1-\tfrac{1}{n})h_0+h_1|<\gb/2\big\}.$$
For $ (q,p)\in \R^*\times\R$, the fact that $\nabla [\frac1n\cL_{\gL_n}]$ is a $\cC^1$ diffeomorphism from $\cD_{\gb,n}$ to $\R^2$ (see   \cite[Lemma 5.4]{CNP16}),
allows us to choose $\bh:=\bh_n(q,p)$ in \eqref{def:Pnh} 
with $\bh_{n}(q,p)$ the unique solution (in $\bh$) of the equation 
 \begin{equation}\label{def:hnq}
\bE_{n,\bh}\Big[\frac1n\gL_n\Big] = \nabla \Big[\frac1n\cL_{\gL_n}\Big](\bh) = (q,p).
\end{equation}
Under  $\Pb_{n,\bh_n(q,p)}$ the event $\{\tfrac1n \Lambda_n= (q,p)\}$ becomes typical and can be sharply estimated.
\smallskip

In the present context though
we only consider events where the random walk returns to the origin after $n$ steps, i.e., $\{\tfrac1n \Lambda_n= (q,0)\}$. Moreover,  for straightforward symmetry reasons
(stated e.g.  in \cite[Remark 5.5]{CNP16}) we have 
$$\bh_n(q,0)=\big(h_n^q, - h_n^q \, (\tfrac12-\tfrac{1}{2n})\big), \quad q\in \R$$
where $h_n^q$ is the unique solution in $h$ of the equation  $\cG'_n(h)=q$,
with
\begin{align}\label{defGn}
\nonumber \cG_n(h)&:= \frac1n \mathcal{L}_{\Lambda_n}\big[h,-\tfrac{h}{2} (1-\tfrac1n)\big]\quad \text{for}\quad h\in  \big(-\tfrac{n\, \beta}{n-1},\tfrac{n\, \beta}{n-1}\big)\\
&=\tfrac1n \sum_{i=1}^n \cL\big[\tfrac{h}{2} (1-\tfrac{2i-1}{n}) \big].
\end{align}
It is proven below in Lemma \ref{diffeoGN} that $h\mapsto \cG_n'(h)$ is a $\cC^1$ diffeomorphism from $(-\tfrac{n\, \beta}{n-1},\tfrac{n\, \beta}{n-1}\big)$ to $\R$ which 
justifies the existence and uniqueness of $h_n^q$.
As a consequence, instead of those tilted probability laws in \eqref{def:Pnh}, we will rather use the probability laws $ \Pb_{n,h}$  that depend on the sole parameter
$h\in (-\tfrac{n\, \beta}{n-1},\tfrac{n\, \beta}{n-1})$, i.e.,
\begin{equation}\label{def:Pnh1}
\frac{\dd \Pb_{n,h}}{\dd \Pbbzero} (X) = e^{\psi_{n,h}(A_{n-1}, X_n)} \,,
\quad \text{with} \quad
\psi_{n,h}(a,x) := h\,\frac{a}{n}-\frac{h}{2} \Big(1-\frac{1}{n}\Big) x -n\,  \cG_n(h)\,,\quad x,a\in\bbZ\,.
\end{equation}
We note that, under $ \Pb_{n,h}$,   
the increments $(X_{i}-X_{i-1})_{i=1}^{n}$ are independent  and $X_{i}-X_{i-1}$ follows the law $\tilde\Pb$ with parameter $\frac{h}{2} (1-\tfrac{2i-1}{n}\big)$ (recall \eqref{def:Pgd}).  

It remains to define the continuous counterpart of $\cL_{\Lambda_n}$.
\begin{equation}\label{def:cDgb}
\cD_\gb \;:=\; \Big\{(h_0,h_1)\in\R^2\,;\;|h_1|<\gb/2,\; |h_0+h_1|<\gb/2\Big\},
\end{equation}
and for every $\bh=(h_0,h_1)\in\cD_\gb$, 
\begin{equation}\label{def:cLgL}
\cL_\gL(\bh) := \int_0^1 \cL(h_0x+h_1)\dd x.
\end{equation}
As stated in \cite[Lemma 5.3]{CNP16}, $\nabla\cL_\gL(\bh)$ that can be written as 
\begin{equation}
\begin{aligned}\label{def:tildebh}
\nabla\cL_\gL(\bh) \;&=\; (\partial_{h_0} \cL_\gL, \partial_{h_1} \cL_\gL)(\bh)\\
&=\; \Big(\int_0^1 x\cL'(xh_0+h_1)\dd x, \int_0^1 \cL'(x h_0+h_1)\dd x\Big),
\end{aligned}
\end{equation}
is a $\cC^1$ diffeomorphism from $\cD_\gb$ to $\R^2$. Thus, for $(q,p)\in \R^2$ we  let $\tilde\bh(q,p)$ 
be the unique solution in $\bh\in\cD_\gb$ of the equation $\nabla\cL_\gL(\bh)=(q,p)$. 
As mentioned above for the discrete case,  since we will only consider the case $p=0$, the fact that $\cL'$ is an odd function combined with \eqref{def:tildebh} ensures us that 
\begin{equation}\label{hrt}
\tilde\bh(q,0)=\Big(\tilde h^q, - \frac{ \tilde h^q}{2} \Big), \quad q\in \R
\end{equation}
where $\tilde h^q$ is the unique solution in $h$ of the equation $ \cG'(h)=q$, with
\begin{align}\label{deftiG}
\nonumber  \cG(h)&:=  \mathcal{L}_{\Lambda}\big(h,-\tfrac{h}{2} \big)\quad \text{for}\quad h\in (-\beta,\beta)\\
&= \int_0^1 \cL(h(\tfrac12-x))\dd x.
\end{align}
With Lemma \ref{diffeotildeG} below,  we prove that  $h\mapsto  \cG'(h)$ is a $\cC^1$ diffeomorphism from $(- \beta,\, \beta)$ to $\R$. This 
justifies the existence and uniqueness of $\tilde h^q$.

\bigskip

\noindent {\bf Outline of the proof of Proposition \ref{lemunif}}
With Step 1 below, we bound from above the difference between the finite size exponential decay rate of 
 $\Pb_{\gb}(\cV_{n,q n^2})$ and its limit as $n$ tends to $\infty$. In Step 2, we divide $\Pb_{\gb}(\cV_{N,qN^2})$
into a main term $M_{N,q}$ and an error term $E_{N,q}$. The main term is obtained by adding to the definition of $\cV_{N,qN^2}$ some constraints concerning the possible values taken by $X_{a_N}, A_{a_N}, X_{N-a_N}$ and $A_{N-a_N-1}$ for an ad-hoc sequence  $(a_N)_{N\geq 1}$ satisfying both $a_N=o(N)$ and $\lim_{n\to \infty}a_N=\infty$.  The $E_{N,q}$ term is bounded above in Step 3, while in Step 4 we provide a sharp estimate of $M_{N,q}$, and we conclude the proof in Step 4
 by computing the pre-factors in the estimate of $\Pb_{\gb}(\cV_{N,qN^2})$.



\subsection{Step 1}\label{Step1}
The aim of this step is to prove the following Proposition, which is a strong improvement of \cite[Proposition 2.3]{CNP16}
since we bound from above the gap between discrete quantities and their continuous counterparts  by $n^{-2}$ instead of $n^{-1}$.
As mentioned in Remark \ref{linkwithWP} above, our proof is close in spirit to that of \cite[Theorem 2]{PerfWach19}, in particular for Proposition \ref{controlderiv} below. Recall the definition of $\tpsi$ from~\eqref{defF}.
\begin{proposition}\label{controlhn}
For $[q_1,q_2]\subset (0,\infty)$, there exists a $C>0$ and an $n_0\in \N$ such that 
for every $n\geq n_0$ and $q\in [q_1,q_2]\cap\frac{\N}{n^2}$
\begin{equation}\label{controllhn}
\left|\frac1n \psi_{n,h_n^q}(qn^2,0)- \tpsi(q)\right|\leq \frac{C}{n^2},
%
\end{equation}
and
\begin{equation}\label{hn}
\big|h_n^q-\tilde h^q\big|
\leq  \frac{C}{n^2}.
\end{equation}
\end{proposition}

\begin{proof}
We recall \eqref{defGn} and \eqref{deftiG} where the definitions of $\cG_n$ and $\cG$ are displayed respectively.
Let us start with two lemmas stating that $\cG_n'$ and $\cG'$ are $\cC^1$-diffeomorphisms.

\begin{lemma}\label{diffeoGN}
For every $n\geq 2$, the application $\cG_n$ is $\cC^\infty$, strictly convex, even and satisfies 
$\cG'_n(h)\to+\infty$ as $h\nearrow\frac{n\gb}{n-1}$. Moreover, there exists $R>0$ such that 
$(\cG_n)''(h)\geq R$ for every $n\geq 2$ and $h\in  (-\tfrac{n\, \beta}{n-1},\tfrac{n\, \beta}{n-1})$.  As a consequence,
$(\cG_n)'$ is an increasing $\cC^1$ diffeomorphism from $(-\tfrac{n\, \beta}{n-1},\tfrac{n\, \beta}{n-1})$ to $\R$ and 
$(\cG_n)'(0)=0$.
\end{lemma}

\begin{lemma}\label{diffeotildeG}
The application $\cG$ is $\cC^\infty$, strictly convex, even and satisfies 
$\cG'(h)\to+\infty$ as $h\nearrow\gb$. Moreover, there exists an $R>0$ such that 
$(\cG)''(h)\geq R$ for  $h\in  (-\beta, \beta)$. As a consequence,
$(\cG)'$ is an increasing $\cC^1$ diffeomorphism from $(- \beta, \beta)$ to $\R$ and 
$(\cG)'(0)=0$.
\end{lemma}
\begin{proof}
Lemma \ref{diffeoGN} being a discrete counterpart of  Lemma \ref{diffeotildeG}, we will only prove the latter here. One easily observes (recall \eqref{def:cL}) that $\cL$ is $\cC^\infty$,  strictly convex,  even and that $\cL''$ is bounded below by a positive constant uniformly on $(-\beta/2,\beta/2)$. With the help of \eqref{deftiG} we state that  $\cG$ enjoys the very same properties on $(-\beta,\beta)$ so that it simply remains to prove that  $\lim_{h\to \beta^-} \cG'(h)=\infty$. To that aim, we compute
\begin{align*}\label{compderL}
\cG'(h)&=2\, \int_0^{1/2}\,  u\,  \cL'(hu)\, du \geq \frac12 \, \int_{1/4}^{1/2} \, \cL'(hu)\, du\\
&\geq \frac1{2h}  \, (\cL(\tfrac{h}{2})-\cL(\tfrac{h}{4}))
\end{align*}
and it suffices to observe that  $\lim_{h\to \beta^-} \cL(h/2)=\infty$ to complete the proof of the lemma.
\end{proof}

Let us resume the proof of Proposition \ref{controlhn} by recalling that $h_n^q$ is the unique solution in $h$ of $(\cG_n)'(h)=q$
and that $\tilde h^q$ is the unique solution in $h$ of $\cG'(h)=q$.  
At this stage, we state a key result which substantially improves \cite[Lemma 5.1]{CNP16}. Its proof is postponed to Appendix~\ref{pr:propderiv}.
\begin{proposition}\label{controlderiv}
For every $K\in (0,\beta)$, there exist a $C_K>0$ and a $n_K\in \N$ such that for $j\in \{0,1\}$
\begin{equation}\label{controldiccont}
\sup_{h\in [-\beta+K,\beta-K]} \big| \cG_n^{(j)}(h)- \cG^{(j)}(h)\big| \leq \frac{C_K}{n^2}, \quad \forall n \geq n_K.
\end{equation}
\end{proposition}
 With the help of Lemma 
\ref{diffeotildeG} we can state the following corollary of Proposition \ref{controlderiv}.

\begin{corollary}\label{coroboundinf}
For every $M>0$, there exists  a $K>0$ such that 
$\cG^{'}(\beta-K)\geq 2 M$ and $\cG^{'}(K)\leq \frac1{2M}$; and there exists an $n_0\in \N$ such that 
 $\cG_n^{'}(\beta-K)\geq M$ and $\cG_n^{'}(K)\leq \frac1{M}$ for every $n\geq n_0$.
\end{corollary}

\begin{remark}\label{rmsc}
A straightforward consequence of Corollary \eqref{coroboundinf} is that for $[q_1,q_2]\subset (0,\infty)$ there exists 
a $K>0$ and an $n_0\in \N$ such that for every $q\in [q_1,q_2]$ and every $n\geq n_0$ we have $h_n^q, \tilde h^q\in [K,\beta-K]$. 
\end{remark}

We now have all the required tools in hand to prove Proposition \ref{controlhn}. 
By using Lemma \ref{diffeotildeG}, Proposition \ref{controlderiv} and Remark \ref{rmsc},  we can state that there exist $K>0$,  $C>0$ and $R>0$ such that for $n$ large enough and $q\in [q_1,q_2]\cap\N/n^2$ we have 
$h_n^q, \tilde h^q \in [K,\beta-K]$, so 
\begin{align}\label{streqh}
 \big | h_n^q-\tilde h^q|\leq \frac{1}{R}\,  \Big| \int_{h_n^q}^{\tilde h^q} \cG''(x)\, dx \Big|&=\frac{1}{R}\,  \big| \cG'(h_n^q)-\cG' (\tilde h^q)\big| \\
\nonumber &=\frac{1}{R}\, \big| \cG'(h_n^q)- \cG'_n (h_n^q)\big| \leq \frac{C}{R\, n^2}
\end{align}
where we have used that $q=\cG'(\tilde h^q)=\cG'_n(h_n^q)$ for the second equality in \eqref{streqh}. The proof of  \eqref{hn} is therefore completed.

It remains to prove \eqref{controllhn}.  To that aim we write  
\begin{align}\label{compfcttaux}
\big| h_n^q \,  q -  \cG_n(h_n^q)- \tilde h^q \,q + \cG(\tilde h^q)\big|\leq U_{n,q}+ q_2 \big |  h_n^q-\tilde h^q \big|,
\end{align}
where  $U_{n,q}:= | \cG_n(h_n^q)- \cG(\tilde h^q)|$.  Proposition \ref{controlderiv} also tells us that there exists a $C>0$ such that  for $n$ large enough and for every $x\in [K,\beta-K]$ we have  $|\cG_n(x)-\cG(x)|\leq C/n^2$. 
Thus, since $\cG$ is $\cC^1$ and recalling Remark~\ref{rmsc} we can write
\begin{align}\label{contrpremt}
\nonumber U_{n,q}&\leq \big| \cG_n(h_n^q)- \cG(h_n^q)\big| + \big | \cG(h_n^q)- \cG(\tilde h_{0}^q) \big|\\
&\leq \frac{C}{n^2}+\sup\big\{|\cG'(x)|, x\in [K,\beta-K]\big\} \big |  h_n^q-\tilde h^q \big|.
\end{align}
At this stage, \eqref{controllhn} is obtained by combining \eqref{compfcttaux} with  \eqref{hn} and \eqref{contrpremt}.  This completes the proof of Proposition \eqref{controlhn}.

\end{proof}

\subsection{Step 2}\label{Step2}

We recall \eqref{def:nu} and \eqref{def:Pnh1} and we set $a_N:=(\log N)^2$. We define the two boxes
\begin{align}\label{defbox}
\mathcal C_N&:=\big[\bE_{N,h_N^q}(X_{a_N})-(a_N)^{3/4}, \bE_{N,h_N^q}(X_{a_N})+(a_N)^{3/4}\big]\\
\nonumber \mathcal D_N&:=\big[\bE_{N,h_N^q}(A_{a_N})-(a_N)^{7/4}, \bE_{N,h_N^q}(A_{a_N})+(a_N)^{7/4}\big]
\end{align}
and rewrite
$\Pb_{\gb}(\cV_{N,qN^2})= M_{N,q}+E_{N,q}$ where
\begin{align}\label{comphnq}
M_{N,q}&:=\bP_{\gb}\big[\cV_{N,qN^2} \cap \big\{X_{a_N}\in \mathcal C_N, A_{a_N}\in \cD_N\big\}\cap \big\{ X_{N-a_N}\in \mathcal C_N, 
A_N-A_{N-a_N-1}\in \cD_N\big\}\Big].
\end{align}
From now on, $M_{N,q}$ will be referred to as the main term and $E_{N,q}$ as the error term.  The proof of Proposition~\ref{lemunif} is a straightforward consequence of Lemmas~\ref{bounderrorthe} and~\ref{contromain} displayed below, which will be proven in Steps 3 and 4, respectively. 

Let us start with the following Lemma, which allows us to control the error therm uniformly in $q$ belonging to any compact set of $(0,\infty)$. 
 \begin{lemma}\label{bounderrorthe}
 For $[q_1,q_2]\subset (0,\infty)$, there exists  $\gep:\N\mapsto \mathbb{R}^+$ such that $\lim_{N\to \infty} \gep(N)=0$ and  for every $N\in \N$ and 
 $ q\in [q_1,q_2]\cap \frac{1}{N^2}$
 \begin{equation}
E_{N,q}\,\leq\, \frac{\gep(N)}{N^2} e^{-N\,\tpsi(q)}.
\end{equation}
 \end{lemma}
 
With the next Lemma, we estimate the main therm uniformly in $q$ belonging to any compact set of $(0,\infty)$. Recall \eqref{defKB} and \eqref{defvartheta} for the definitions of $\kappa$ and $\vartheta$.
\begin{lemma}\label{contromain}
For $\beta>0$ and $[q_1,q_2]\subset (0,\infty)$,  
\begin{align}\label{contromain2}
M_{N,q}=  \kappa\big(\tfrac{\tilde h^q}{2}\big)^2\  \frac{\big(\vartheta(\tilde h^q)\big)^{-\frac{1}{2}}}{2\pi N^2} \    e^{-N\,\tpsi(q)}\   (1+o(1))
\end{align}
where $o(1)$ is a function that converges to $0$  as $N \to \infty$  uniformly in  $ q\in [q_1,q_2]\cap \frac{1}{N^2}$.
\end{lemma}

 \subsection{Step 3: proof of Lemma \ref{bounderrorthe}}\label{Step3}

Before starting the proof, let us make a quick remark about the time-reversibility 
of the random walk $X$ of law $\bP_{n,h}$.
\begin{remark}[Time-reversal property]\label{timerev}
Recall the definition of $\tilde \Pb_{h}$ from \eqref{def:Pgd}. For $h\in (-\beta/2,\beta/2)$, one can easily check that if $Z$ is a random variable of law $\tilde\Pb_{-h}$ then $-Z$ has law  $\tilde\Pb_{h}$.
Moreover, as explained below  \eqref{def:Pnh1}, if the $n$-step random walk $X:=(X_i)_{i=0}^n$ has law $\bP_{n,h}$, then its increments $(X_{i}-X_{i-1})_{i=1}^{n}$ are independent  and for $i\in \{1,\dots,n\}$ the law of $X_{i}-X_{i-1}$ is $\tilde\Pb$ with parameter $\frac{h}{2} (1-\tfrac{2i-1}{n})$. 
 A first consequence is that $X$ is time-reversible, i.e., 
\begin{equation}\label{timereveq}
(X_i)_{i=0}^n\underset{\text{Law}}{=}(X_{n-i}-X_n)_{i=0}^n.
 \end{equation}
A second consequence is that, under $\bP_{n,h}$, the random walk $(X_i)_{i=0}^n$ is an inhomogeneous Markov chain which,  for $j\in \{1,\dots, N-1\}$, $y\in \mathbb{Z}$ and $\mathcal{O}\subset \mathbb{Z}^{n-j-1}$ satisfies that
\begin{equation}\label{timereveq1}
\bP_{n,h}\big((X_{j+i})_{i=1}^{n-j-1}\in \mathcal{O},\  X_n=0\mid X_j=y\big)=\bP_{n,h}\big((X_{n-j-i})_{i=1}^{n-j-1}\in \mathcal{O},\ X_{n-j}=y\big).
\end{equation}
Note finally that the case $h=0$ corresponds to the random walk $X$ with i.i.d. increments of law $\bP_{\gb}$.   


\end{remark}

\medskip

A straightforward application of \eqref{timereveq} with $n=N$ and $h=0$ allows us  to bound the error term from above as 
\begin{equation}\label{boundBNq}
E_{N,q}\leq 2\,  \bP_{\gb}\big(\cV_{N,qN^2}\cap  \big\{X_{a_N}\notin \mathcal C_N\}\big)  + 2\,  \bP_{\gb}\big(\cV_{N,qN^2}\cap  \big\{A_{a_N}\notin \cD_N\}\big) .
\end{equation}
Using \eqref{def:Pnh1} we obtain that for $\cB=\{X_{a_N}\notin \mathcal C_N\}$ or $\cB=\{A_{a_N}\notin \cD_N\}$
\begin{equation}\label{eq:intcon}
 \bP_{\gb}\big(\cV_{N,qN^2}\cap \cB\big)\leq e^{-\psi_{N,h_N^q}(qN^2,0)}\  \Pb_{N,h_N^{q}}(\cV_{N,qN^2}\cap \cB).
%
 \end{equation}
Note that, the first inequality in Proposition \ref{controlhn} allows us to replace $\psi_{N,h_N^q}(qN^2,0)$ with $N\tpsi(q)$ in the exponential of the  r.h.s. in \eqref{eq:intcon}, at the cost of an at most constant factor. Therefore, the proof of Lemma \ref{bounderrorthe} is completed by the following Claim.

 \begin{claim}\label{errorth}
 For $[q_1,q_2]\subset (0,\infty)$, there exists  $\gep:\N\mapsto \mathbb{R}^+$ such that $\lim_{N\to \infty} \gep(N)=0$ and  for every $N\in \N$ and 
 $ q\in [q_1,q_2]\cap \frac{1}{N^2}$
 \begin{equation}\label{twotb}
 \Pb_{N,h_N^{q}}(\cV_{N,qN^2}\cap \{X_{a_N}\notin \mathcal C_N\})+ \Pb_{N,h_N^{q}}(\cV_{N,qN^2}\cap \{A_{a_N}\notin \cD_N\})\leq \frac{\gep(N)}{N^2}.
 \end{equation}
 \end{claim}
 
 \begin{proof}
 
 Let us prove that \eqref{twotb} holds true for
 $R_{N,q}:=\Pb_{N,h_N^{q}}(\cV_{N,qN^2}\cap \{X_{a_N}\notin \mathcal C_N\})$ and for $S_{N,q}:=\Pb_{N,h_N^{q}}(\cV_{N,qN^2}\cap \{A_{a_N}\notin \mathcal D_N\})$.

We set $\{X_{[j,k]}>0\}:=\{X_i>0, j\leq i\leq k\}$ for $j\leq k\in \N$. We develop $R_{N,q}$ depending on the values $y$ and $z$ taken by $X_{a_N}$ and $A_{a_N}$ respectively. Then we use Markov property at time $a_N$, combined with the time reversal property \eqref{timereveq1} with $n=N$, $h=h_N^q$,  $j=a_N$ and 
 \[\mathcal{O}=\Big\{x\in \N^{N-j-1}\colon \sum_{i=1}^{N-j-1} x_i=qN^2-z\Big\},\]
on the time interval $[a_N,N]$ to obtain
 \begin{align}\label{defrnq}
R_{N,q}=\sum_{y\in \mathbb{N}\setminus \mathcal C_N}\sum_{z\in \N} &\,\Pb_{N,h_N^{q}}( X_{a_N}=y,\,   A_{a_N}=z,\,  X_{[1,a_N]}>0)\\
\nonumber     &\quad\times\Pb_{N,h_N^{q}}(X_{[1,N-a_N]}>0, X_{N-a_N}=y, A_{N-a_N-1}=q N^2-z).
 \end{align}
 Let $R^1_{N,q}:=R_{N,q}(|y|\leq N/a_N, |z|\leq N^2/a_N)$, $R^2_{N,q}:=R_{N,q}(|y|> N/a_N)$ and $R^3_{N,q}:=R_{N,q}(|z|> N^2/a_N)$, where $R_{N,q}(A)$ denotes the sum from \eqref{defrnq} restricted to terms satisfying the condition $A$; so that
 \begin{equation}\label{decompornq}
 R_{N,q} \;\leq\; R^1_{N,q} + R^2_{N,q} + R^3_{N,q}\;.
 \end{equation} 
Let us prove the upper bound \eqref{twotb} for all three terms in the r.h.s., starting with $R^1_{N,q}$.
 
\begin{lemma}\label{lem:errorth}
For $[q_1,q_2]\subset (0,\infty)$, there exists  a $C>0$ such that for every $N\geq 1$
and $ q\in [q_1,q_2]\cap \frac{1}{N^2}$ and $(y,z)\in \mathbb{Z}^2$ with $|y|\leq N/a_N$, $|z|\leq N^2/a_N$,
\begin{equation}\label{bounduniftilt}
\Pb_{N,h_N^{q}}(  X_{N-a_N}=y, A_{N-a_N-1}=q N^2-z)\leq \frac{C}{N^2}.
\end{equation}
\end{lemma}
 
 We prove this Lemma afterwards. Plugging this in the development \eqref{defrnq} and dropping the condition $X_{[1,a_N]}>0$, we obtain 
\begin{equation}\label{defrnq2}
R^1_{N,q}\leq \frac{C}{N^2} \Pb_{N,h_N^{q}}( X_{a_N} \notin \mathcal C_N).
\end{equation}
Recalling \eqref{def:Pnh1}, a straightforward computation gives us that 
\begin{equation}\label{eq:varrnq}\text{Var}_{N,h_N^{q}}\big(X_{a_N}\big)=\sum_{i=1}^{a_N} \cL''\Big(\tfrac{h_N^q}{2} (1-\tfrac{2i-1}{N})\Big),\end{equation}
 so that, by using Tchebychev inequality  we obtain 
 \begin{align}\label{thcheb}
\nonumber \Pb_{N,h_N^{q}}( X_{a_N} \notin \mathcal C_N)&=  \Pb_{N,h_N^{q}}\big(|X_{a_N}-\bE_{N,h_N^q}(X_{a_N})|>(a_N)^{3/4}\big)\\
\nonumber &\leq \frac{1}{a_N^{3/2}}\text{Var}_{N,h_N^{q}}\big(X_{a_N}\big)\\
\nonumber &  \leq  \frac{1}{a_N^{1/2}} \ \sup_{x\in [0,1]} \Big|\cL''\Big(h_{N}^q (\tfrac12-x)\Big)\Big|\\
& \leq \frac{C_1}{\sqrt{a_N}}
 \end{align}
 for some uniform $C_1>0$, where the last inequality is a consequence of Remark \ref{rmsc} and the observation that $\cL$ is $\cC^\infty$ on $(-\beta/2,\beta/2)$. It remains to combine \eqref{defrnq2} and \eqref{thcheb} to complete the proof for $R^1_{N,q}$.
 
Regarding $R^2_{N,q}$ 
let $(U_i)_{i\geq1}$ be the increments of the process $X$ ; recalling \eqref{def:Pnh1} and applying a Chernov inequality for some $\lambda>0$ small, we have
\begin{align}\label{eq:lem:errorth:1} \nonumber &\sum_{y> N/a_N}\bP_{N,h_N^q}(X_{a_N}=y,A_{a_N}=z, X_{[1,a_N]}>0) \\\nonumber &\;\;\leq\; \Pb_{N,h_N^q}(X_{a_N}>N/a_N)\\
\nonumber &\;\; \leq \, \Ebbzero\left[\exp\left(\lambda(X_{a_N}-N/a_N) + \sum_{i=1}^{a_N} U_i\Big(\frac{h_N^q}2\Big(1-\frac{2i-1}N\Big)\Big)-\cL\Big(\frac{h_N^q}2\Big(1-\frac{2i-1}N\Big)\Big)\right)\right]\\
&\;\; \leq \, \,e^{-\lambda N/a_N} e^{c \,a_N} \;\leq\; C\, e^{-\frac{\gl N}{2a_N}}\,,
\end{align}
for some $c>0, C>0$, which can be taken uniformly in $q\in[q_1,q_2]$. The same upper bound holds for the sum over $y<N/a_N$, hence
\begin{align}\label{eq:lem:errorth:2}
\nonumber R^2_{N,q} \;&\leq\; \sum_{z\in\bbZ} \,\bP_{N,h_N^q}(A_{N-a_N-1}=qN^2-z) \\
\nonumber &\qquad\qquad\times \sum_{|y|> N/a_N}\bP_{N,h_N^q}(X_{a_N}=y,A_{a_N}=z, X_{[1,a_N]}>0)\\
&\leq\; 2\,C\, e^{-\frac{\gl N}{2a_N}}
\end{align}

The last term $R^3_{N,q}$ can be handled similarly, by noticing that $A_{a_N}>N^2/a_N$ implies $X_i>N^2/a_N^2$ for some $1\leq i \leq a_N$, thereby
\begin{align*}
&\sum_{z> N^2/a_N}\bP_{N,h_N^q}(X_{a_N}=y,A_{a_N}=z, X_{[1,a_N]}>0)  \\
&\qquad\leq\; \bP(A_{a_N}>N^2/a_N)\;\leq\; \sum_{i=1}^{a_N} \bP(X_i > N^2/a_N^2) \;\leq\; C\,a_N e^{-\frac{\lambda N^2}{2a_N^2}}\,,
\end{align*}
where we reproduced \eqref{eq:lem:errorth:1}. Similarly to \eqref{eq:lem:errorth:2}, we obtain an upper bound on $R^3_{N,q}$, and recollecting \eqref{decompornq} with \eqref{defrnq2}, \eqref{thcheb} and \eqref{eq:lem:errorth:2}, this finally proves \eqref{twotb} for $R_{N,q}$.

Let us now consider $S_{N,q}$, which we develop similarly to \eqref{defrnq}. Notice that the term $S_{N,q}(y\notin\cC_N)$ is already bounded from above by $R_{N,q}$, and that $S_{N,q}(|z|>N^2/a_N)$ follows the same upper bound as $R^3_{N,q}$ (we do not replicate the proof), thus we only have to prove \eqref{twotb} for $S^1_{N,q}:=S_{N,q}(y\in\cC_N,|z|\leq N^2/a_N)$ to complete the proof of Claim~\ref{errorth}. Recalling Lemma~\ref{lem:errorth}, we have
\begin{equation}\label{defsnq2}
S^1_{N,q}\leq \frac{C}{N^2} \Pb_{N,h_N^{q}}( A_{a_N} \notin \cD_N).
\end{equation}
Similarly to  \eqref{eq:varrnq}, a direct computation gives
\[\text{Var}_{N,h_N^{q}}\big(A_{a_N}\big)=\sum_{i=1}^{a_N} (a_N+1-i)^2\cL''\Big(\tfrac{h_N^q}{2} (1-\tfrac{2i-1}{N})\Big),\]
and using Tchebychev inequality, we obtain
 \begin{align}
\nonumber \Pb_{N,h_N^{q}}( A_{a_N} \notin \mathcal D_N) &\leq \frac{1}{a_N^{7/2}}\text{Var}_{N,h_N^{q}}\big(A_{a_N}\big)\\
\nonumber &  \leq  \frac{1}{a_N^{1/2}} \ \sup_{x\in [0,1]} \Big|\cL''\Big(h_{N}^q (\tfrac12-x)\Big)\Big|\\
\nonumber & \leq \frac{C_1}{\sqrt{a_N}}
 \end{align}
 for some uniform $C_1>0$, where the last inequality is a consequence of Remark \ref{rmsc} and the observation that $\cL$ is $\cC^\infty$ on $(-\beta/2,\beta/2)$. Combining this with \eqref{defsnq2}, this concludes the proof of the Claim.
 \end{proof}

\begin{proof}[Proof of Lemma~\ref{lem:errorth}]
To lighten upcoming formulae in this proof, we will write $N_1:=N-a_N$. 
Recall \eqref{def:Pnh1}, and that $\Pb_{\beta,y}$ denotes the law of the random walk starting from $y\in\bbZ$ with increments distributed as $\Pb_\beta$. Summing over possible values for $(X_N,A_{N-1})$ and using Markov property, we write
\begin{align}\label{eq:lem:errorth:1.5}
\nonumber&\Pb_{N,h_N^{q}}(X_{N_1}=y, A_{N_1-1}=q N^2-z) \\\nonumber&\;\;= \sum_{(u,v)\in\bbZ^2} \Pb_{N,h_N^{q}}(X_{N_1}=y, A_{N_1-1}=q N^2-z, X_N=u, A_{N-1}=v)\\
\nonumber&\;\;= \sum_{(u,v)\in\bbZ^2} e^{\psi_{N,h_N^q}(v,u)} \Pbbzero(X_{N_1}=y, A_{N_1-1}=q N^2-z, X_N=u, A_{N-1}=v)\\
\nonumber&\;\;= \sum_{(u,v)\in\bbZ^2} e^{\psi_{N,h_N^q}(v,u)} \Pbbzero(X_{N_1}=y, A_{N_1-1}=q N^2-z)\\
\nonumber&\hspace{5cm}\times\Pb_{\beta,y}(X_{a_N}=u, A_{a_N-1}=v-qN^2+z-y)\\
\nonumber&\;\;= \Pb_{N_1,h_{N_1}^q}(X_{N_1}=y, A_{N_1-1}=q N^2-z,) \!\sum_{(u,v)\in\bbZ^2} \exp\left(\psi_{N,h_N^q}(v,u) - \psi_{N_1,h_{N_1}^q}(qN^2-z,y)\right) \\
&\hspace{5cm}\times\Pb_{\beta,y}(X_{a_N}=u, A_{a_N-1}=v-qN^2+z-y).
\end{align}
A straightforward consequence of Proposition~\ref{llta} and Remark~\ref{remhess} from Appendix~\ref{pr:estimG} is that the first factor is uniformly controlled by $\frac C{N^2}$ for some uniform $C>0$, so it remains to prove that the second factor is uniformly bounded. It can be written as
\[
\Ebbzero\left[\exp\left(\psi_{N,h_N^q}\left(A_{a_N-1} +qN^2 -z+y\,a_N,\, X_{a_N}+y \right) - \psi_{N_1,h_{N_1}^q}(qN^2-z,y)\right)\right],
\]
and can be estimated with Proposition~\ref{controlhn}. Furthermore, we claim that for all $q\in\R$,
\begin{equation}\label{eq:LGqrelation}
\cL\big(\tilde h^q/2\big) - q\hspace{1pt}\tilde h^q - \cG\big(\tilde h^q\big) \,=\, 0\,.
\end{equation}
Indeed,
\begin{align*}
\cL\big(\tilde h^q/2\big) - \cG\big(\tilde h^q\big)\,&=-\int_0^1 \cL\big(\tilde h^q(\tfrac12-y)\big)\, dy+\cL(\tilde h^q/2)\\
&=\,\int_0^1 \int_y^1\tilde h^q \cL'\big(\tilde h^q(u-\tfrac12)\big) du\, dy\\
&=\,\tilde h^q\int_0^1 u \cL'\big(\tilde h^q(u-\tfrac12)\big) du\\
&=\, \tilde h^q \hspace{1pt} \cG'\big(\tilde h^q\big) \,=\, \tilde h^q \hspace{1pt}q\,,
\end{align*}
where we have used that $\cL$ is even to obtain the second line and that $\tilde h^q$ is solution in $h$ of 
$\cG'(h)=q$ (recall Step 1) to obtain the last line. 

We deduce from \eqref{eq:lem:errorth:1.5}, Proposition~\ref{controlhn} and \eqref{eq:LGqrelation} (we do not write the details here) that there exists $C_1>0$ such that for $n\in\N$, $h\in]-\beta,\beta[$ and $|y|\leq N/a_N$, $|z|\leq N^2/a_N$,
\begin{align}\label{eq:lem:errorth:3}
\nonumber&\Pb_{N,h_N^{q}}(X_{N_1}=y, A_{N_1-1}=q N^2-z) \\
& \qquad \leq\, \frac {C_1}{N^2} e^{-a_N\cL(\tilde h^q/2)} \,\Ebbzero\left[\exp\left(2\tilde h^q\frac{A_{a_N-1}}{N} - \frac{\tilde h^q}{2} X_{a_N}\left(1-\frac{2}{N} \right) \right)\right].
\end{align}
Let us now focus on that last factor, which we part into two terms. On the one hand, notice that
\begin{align}\label{eq:lem:errorth:4}
\nonumber&\Ebbzero\left[\exp\left(2\tilde h^q\frac{A_{a_N-1}}{N} - \frac{\tilde h^q}{2} X_{a_N}\left(1-\frac{2}{N} \right) \right)\ind_{\{A_{a_N-1}\leq N\}}\right] \\
\nonumber&\qquad\leq\; e^{2\tilde h^q}\Ebbzero\left[\exp\left(- \frac{\tilde h^q}{2} X_{a_N}\left(1-\frac{2}{N} \right) \right)\right]\\
&\qquad\leq\; e^{2\tilde h^q} \exp\left(a_N \cL\Big(\frac{\tilde h^q}2\Big(1-\frac2N\Big)\Big)\right)
\;\leq\; e^{2\tilde h^q + a_N \cL(\tilde h^q/2)}\,,
\end{align}
where we used that $\cL$ is symmetric, and non-decreasing on $[0,+\infty)$. On the other hand, the event $A_{a_N-1}>N$ implies that $U_i>N/a_N^2$ for some $1\leq i \leq a_N-1$, where $(U_i)_{i\geq1}$ denotes the increments of the process $X$. So
\begin{align}\label{eq:lem:errorth:5}
\nonumber&\Ebbzero\left[\exp\left(2\tilde h^q\frac{A_{a_N-1}}{N} - \frac{\tilde h^q}{2} X_{a_N}\left(1-\frac{2}{N} \right) \right)\ind_{\{A_{a_N-1}> N\}}\right] \\
\nonumber&\qquad\leq\; \sum_{i=1}^{a_N-1} \Ebbzero\left[\ind_{\{U_i>N/a_N^2\}} \exp\Bigg(\sum_{j=1}^{a_N-1} U_j\,\tilde h^q \Big(\tfrac{2(a_N-i)+1}N-\tfrac12\Big)\Bigg)\right]\\
\nonumber&\qquad\leq\; \sum_{i=1}^{a_N-1} \left[\prod_{j\neq i} e^{\cL\big(\tilde h^q\big(\frac{2(a_N-i)+1}{N}-\frac12\big)\big)} \times \sum_{k> N/a_N^2} \frac{e^{-\frac{\gb}2k}}{c_\gb} e^{\tilde h^q\big(\frac{2(a_N-k)+1}{N}-\frac12\big)} \right]\\
&\qquad\leq\; C\, (a_N-1) \, e^{C_1(a_N-2)} \, e^{-\tfrac{\gb}{2}\tfrac{N}{a_N^2}} \;\leq\; C_2\,e^{-\tfrac{\gb}{4}\tfrac{N}{a_N^2}}\,,
\end{align}
for some $C,C_1,C_2>0$ uniform in $q\in[q_1,q_2]$. Plugging \eqref{eq:lem:errorth:4} and \eqref{eq:lem:errorth:5} in \eqref{eq:lem:errorth:3}, this concludes the proof of Lemma~\ref{lem:errorth}.
\end{proof}

\subsection{Step 4: proof of lemma \ref{contromain}}\label{Step4}

In the following, we set $\bar x=(x_1,x_2)$ and $\bar a=(a_1,a_2)$. We define also
\begin{equation}
\mathcal{H}_N:=\{(\bar x,\bar a)\in \mathcal C_N^2\times \mathcal D_N^2\}.
\end{equation} 

We use Markov property at time $a_N$ and $N-a_N$ and we apply time-reversibility on the part of the walk between time $N-a_N$ and $N$ (i.e.,  \eqref{timereveq} with $n=a_N$ and $h=0$) to obtain
\begin{equation}\label{decomp}
M_{N,q}=\sum_{(\bar x,\bar a)\in \cH_N} R_{N}(x_1,a_1) \, T_N(\bar x, \bar a)\, R_N(x_2,a_2),
\end{equation}
with 
\begin{align}
R_{N}(x,a):= \bP_{\gb}(X_{[1,a_N]}>0,  X_{a_N}=x, A_{a_N}=a)
\end{align}
and, after setting $N_2=N-2a_N$,
\begin{equation}
T_N(\bar x, \bar a):= \bP_{\gb}\big(X_{N_2}=x_2-x_1,\, A_{N_2-1}=qN^2-a_1-a_2-x_1 (N_2-1),\,  X_{[0,N_2]}>-x_1\big)
\end{equation}
We begin with $R_{N}(x_1,a_1)$ and $R_N(x_2,a_2)$. For conciseness, we set $ h:=\tilde h^q/2$. Then, we recall \eqref{def:Pgd} and we perform a change of measure
by tilting every increment $(X_{i+1}-X_i)_{i=0}^{a_N-1}$ with $\tilde \Pb_h$, that is for $(x,a)\in \{(x_1,a_1), (x_2,a_2)\}$
\begin{equation}\label{tiltpremtrois}
R_{N}(x,a):= e^{-h x+a_N \cL(h)}\,  \tilde\Pb_h(X_i>0, \forall i\leq a_N, X_{a_N}=x, A_{a_N}=a).
\end{equation}
For the second term $T_N(\bar x, \bar a)$ we apply the tilting introduced in \eqref{def:Pnh1} and we obtain
\begin{align}\label{expT}
T_N(\bar x, \bar a)=G_{N,\bar x,\bar a}^{\,q}\  e^{-h_{N_2}^q \big(\frac{qN^2-a_1-a_2}{N_2}-x_1\big(1-\frac1{N_2}\big)\big)} \, e^{+\frac{h_{N_2}^q}{2} \big(1-\frac{1}{N_2}\big) (x_2-x_1)+N_2\, \cG_{N_2}(h_{N_2}^q)}
\end{align} 
with
\begin{equation}\label{defG}
G_{N,\bar x,\bar a}^{\,q}:=\Pb_{N_2, h_{N_2}^{q}}\Big[\Lambda_{N_2}=\big(\tfrac{qN^2-a_1-a_2}{N_2}-x_1\big(1-\tfrac1{N_2}\big), \, x_2-x_1\big), X_{[0,N_2]}>-x_1\Big].
\end{equation}
Recollecting \eqref{decomp}, \eqref{tiltpremtrois}, \eqref{expT} and \eqref{defG}, we obtain 
\begin{align}\label{decomp1}
M_{N,q}= \sum_{(\bar x,\bar a)\in \cH_N}  e^{B^1_{N,q}(\bar a)+B^2_{N,q}(\bar x)} \tilde\Pb_h(&X_{[1,a_N]}>0, X_{a_N}=x_1, A_{a_N}=a_1)\  G_{N,\bar x,\bar a}^{\,q}\\
\nonumber  &  \tilde \bP_{h}(X_{[1,a_N]}>0, X_{a_N}=x_2, A_{a_N}=a_2)
\end{align}
with 
\begin{align}
B^1_{N,q}(\bar a)\,&:=\,N_2\, \cG_{N_2}(h_{N_2}^q)+2\, a_N\, \cL(h)-h_{N_2}^q \, \tfrac{qN^2-a_1-a_2}{N_2}\,,\\
\nonumber B^2_{N,q}(\bar x)&:=\, (x_1+x_2)\bigg[\tfrac{h_{N_2}^q}{2} \big(1-\tfrac{1}{N_2}\big) -h\bigg] \,.
\end{align}
Henceforth, we drop the $\bar a$ and $\bar x$ dependency of $B^1_{N,q}(\bar a)$ and $B^2_{N,q}(\bar x)$ for conciseness.  
We consider first  $B^2_{N,q}$. We recall that $h=\tilde h^q/2$, 
thus
\begin{align}
B^2_{N,q}=\frac{x_1+x_2}{2} \, \Big(h_{N_2}^q -\tilde h^q - \tfrac{h_{N_2}^q}{N_2}\Big) 
\,,
\end{align} 
which allows us to write the upper bound
\begin{align}
|B^2_{N,q}|\leq \frac{|x_1|+|x_2|}{2} \  \Big(\big |h_{N_2}^q -\tilde h^q\big |+ \tfrac{ h_{N_2}^q}{ N_2}\Big)\,.
\end{align} 
We recall \eqref{def:Pgd} and observe that $x\mapsto  \tilde \bE_{x}(X_1)$ is non decreasing. Thus, it suffices to apply Remark
\ref{rmsc} to conclude that there exist a $K>0$ and a $N_0\in \N$ such that, for $N\geq N_0$ and $q\in [q_1,q_2]\cap \frac{\N}{N^2}$  the inequality  
$\bE_{N,h_N^q}(X_{a_N})\leq a_N \,  \tilde \bE_{(\beta-K)/2}(X_1)$
holds true. Therefore, by recalling the definition of $\mathcal C_N$ in \eqref{defbox} we can assert that every $(x_1,x_2)\in \mathcal (C_N)^2$ 
satisfies  $|x_1|+|x_2|\leq \cst  a_N$.
It remains to use \eqref{hn} to conclude that provided  $N_0\in \N$ is large enough,  for every  $q\in [q_1,q_2]\cap \frac{\N}{N^2}$ and  $N\geq N_0$
\begin{align}\label{bdB2}
\nonumber |B^2_{N,q}|\;&\leq\, \text{(const.)}\,   a_N\,  \Big(\frac{C}{(N_2)^2}+\frac{\beta}{N_2}\Big) \\
&\leq\, \text{(const.)}\,  \frac{a_N}{N_2},
\end{align} 
and therefore, $B^2_{N,q}$ converges to $0$ as $N\to \infty$ uniformly in $q\in [q_1,q_2]$. 

Now, we consider $B^1_{N,q}$. We recall the definition of $\mathcal D_N$ in \eqref{defbox} and similarly to what we did above, we can assert that
provided $N_0$ is chosen large enough, for $N\geq N_0$ and $q\in [q_1,q_2]\cap \frac{\N}{N^2}$  we have that  every $(a_1,a_2)\in \mathcal (D_N)^2$
satisfies $|a_1|+|a_2|\leq \cst a_N^2$. Moreover,  
\begin{equation}\label{perturbq}
\frac{qN^2-a_1-a_2}{N_2}=q N+2 q a_N+O\Big(\frac{(a_N)^2}{N}\Big)
\end{equation}
and therefore  $B^1_{N,q}=\hat B_N(q) +O\big(\frac{(a_N)^2}{N}\big)$ with
\begin{equation}\label{defbnq}
\hat B_N(q) :=N_2\, \cG_{N_2}(h_{N_2}^q)+2\, a_N\, \cL(\tilde h^q/2)-h_{N_2}^q \, q (N+2\, a_N).
\end{equation}
Our aim is to show that 
\begin{equation}\label{bnq}
\hat B_N(q)=N  \cG(\tilde h^q)-N\tilde h^q+o(1)
\end{equation}
and for that we rewrite 
\begin{align*}
\hat B_N(q)&=N_2 \Big[ \cG_{N_2}(h_{N_2}^q)-h_{N_2}^q \, q\Big] +2\, a_N\, \big[\cL(\tilde h^q/2)-2\,  h_{N_2}^q \, q\big] \\ 
 &=N_2 \Big[\cG(\tilde h^q)-\tilde h^q \, q+O\big(\tfrac{1}{(N_2)^2}\big) \Big]+ 2\, a_N\, \Big[\cL(\tilde h^q/2)-2\, \tilde h^q \, q + O\big(\tfrac{1}{(N_2)^2}\big)\Big] \\
&=N \big[\cG(\tilde h^q)-\tilde h^q \, q\big] +2 a_N \big[\cL(\tilde h^q/2) -\tilde h^q \, q- \cG(\tilde h^q)\big] +O(\tfrac{1}{N_2})
\end{align*}
which, by  \eqref{controllhn} and \eqref{hn}, is true uniformly in $N\geq N_0$ and $q\in [q_1,q_2]\cap \frac{\N}{N^2}$, provided $N_0$ is chosen large enough. Recalling \eqref{eq:LGqrelation}, this completes the proof of \eqref{bnq}.


We recall that to lighten notations we sometimes use $h=\tilde h^q/2$. At this stage, we deduce from \eqref{decomp1}, \eqref{bdB2} and \eqref{bnq} that 
\begin{align}\label{decomp2}
M_{N,q}= (1+o(1)) \, e^{N \big(\cG(\tilde h^q)-\tilde h^q q\big)} \sum_{(\bar x,\bar a)\in \cH_N} &\tilde\Pb_h(X_{[1,a_N]}>0, X_{a_N}=x_1, A_{a_N}=a_1)\\
\nonumber &  G_{N,\bar x,\bar a}^{\,q} \   \tilde \bP_{h}(X_{[1,a_N]}>0, X_{a_N}=x_2, A_{a_N}=a_2).
\end{align}
We will complete the present step subject to Lemma \ref{estimG} below. This Lemma will be proven in Appendix~\ref{pr:estimG}.
\begin{lemma}\label{estimG}
For $\beta>0$,
\begin{equation}
 G_{N,\bar x,\bar a}^{\,q}=\frac{\big(\vartheta(\tilde h^q)\big)^{-\frac{1}{2}}}{2\pi N^2} (1+o(1))
\end{equation}
where $o(1)$ is a function that converges to $0$ as $N\to \infty$ uniformly in 
$q\in [q_1,q_2]\cap \frac{\N}{N^2}$ and $(\bar x,\bar a)\in \cH_N$.
\end{lemma}
By applying Lemma \eqref{estimG}, we may rewrite \eqref{decomp2} as   
\begin{align}\label{hatanq}
M_{N,q}=  (1+o(1)) \, \frac{\big(\vartheta(\tilde h^q)\big)^{-\frac{1}{2}}}{2\pi N^2}\,  e^{N \big(\cG(\tilde h^q)-\tilde h^q q\big)} \, \hat M_{N,q}
\end{align}
with
\begin{align}
\nonumber \hat M_{N,q}&:=\!\!\sum_{(\bar x,\bar a)\in \cH_N}\!\! \tilde\Pb_h(X_{[1,a_N]}>0, X_{a_N}=x_1, A_{a_N}=a_1)\,   \bP_{h}(X_{[1,a_N]}>0, X_{a_N}=x_2, A_{a_N}=a_2)\\
&= \big[ \tilde\Pb_h(X_{[1,a_N]}>0, X_{a_N}\in \mathcal{C}_N, A_{a_N}\in \mathcal{D}_N)\big]^2,
\end{align}
from which we deduce that 
\begin{align}\label{estimhatanq}
\big|\big(\hat M_{N,q}\big)^{1/2}-  \tilde\Pb_h(X_{[1,a_N]}>0)\big| \leq \tilde\Pb_h(X_{a_N}\notin \mathcal{C}_N)+\tilde\Pb_h(A_{a_N}\notin \mathcal{D}_N).
\end{align}

Remark \ref{rmsc} guarantees us again
that there exist a $K>0$ and a $N_0\in \N$ such that  $\tilde h^q\in [K,\beta-K]$ for  $q\in [q_1,q_2]$ and that $\tilde h_N^q\in [K,\beta-K]$ for $N\geq N_0$ and $q\in [q_1,q_2]\cap \frac{N}{N^2}$.  Since
$\{X_{a_N}\notin \mathcal{C}_N\}=\{|X_{a_N}-\bE_{N,h_N^q}(X_{a_N})|\geq (a_N)^{3/4}\}$
we set  $\alpha_q:=\tilde \bE_h(X_1)$ so that $\tilde \bE_h(X_{a_N})=a_N \alpha_q$ (with $h=\tilde h^q/2$)
and we compute 
\begin{align}
\nonumber \big| a_N \alpha_q- \bE_{N,h_N^q}(X_{a_N})\big |&=\Big| \sum_{i=1}^{a_N} \tilde \bE_h(X_1)-\tilde \bE_{h_{N}^q (1-\frac{i}{N})-\frac{h_{N}^q}{2} (1-\frac{1}{N})}(X_1)\Big |\\
\nonumber &\leq \sum_{i=1}^{a_N}\Big| \cL'(\tfrac{\tilde h^q}{2})- \cL'\Big(\tfrac{h_{N}^q}{2}+ \tfrac{h_{N}^q}{2} (\tfrac{1-2i}{N})\Big)\Big |\\
\nonumber &\leq  \sum_{i=1}^{a_N} \max_{x\in \big[\frac K2,\frac{\beta-K}{2}\big] }\big|\cL''(x)\big| \  \Big(\Big| \tfrac{\tilde h^q-h_{N}^q}{2}\Big|+\beta \tfrac{a_N}{N}\Big)\\
&\leq \cst \ \Big(\frac{a_N}{N^2}+\frac{(a_N)^2}{N}\Big).
\end{align}
As a consequence $\{X_{a_N}\notin \mathcal{C}_N\}\subset\{|X_{a_N}-a_N \alpha_q|\geq \tfrac12 (a_N)^{3/4}\}$ and by using Tchebychev inequality we conclude that 
\begin{align}\label{tchebC}
\nonumber \tilde\Pb_h(|X_{a_N}-a_N \alpha_q|\geq \tfrac12 (a_N)^{3/4})&\leq \cst\,  \tfrac{1}{\sqrt{a_N}} \, \text{Var}_{\tilde h^q/2}(X_1)\\
&\leq \cst \, \tfrac{1}{\sqrt{a_N}},
\end{align}
where we have used that $\cL$ is $\mathcal C^2$ and that $\cL^{''}(x)=\text{Var}_x(X_1)$ for every $x\in (-\beta/2,\beta/2)$. This finally proves that  $\lim_{N\to \infty}  \tilde\Pb_h(X_{a_N}\notin \mathcal{C}_N)=0$  
and the same type of argument also gives that   $\lim_{N\to \infty} \tilde\Pb_h(A_{a_N}\notin \mathcal{D}_N)=0$. Both convergences hold true 
uniformly in $q\in [q_1,q_2]$.
\medskip

We come back to \eqref{estimhatanq} and we can write that  
\begin{equation}\label{mnpc}
\hat M_{N,q}= (\tilde\Pb_h(X_{[1,a_N]}>0))^2 (1+o(1))
\end{equation}
 where $o(1)$ is uniform in 
$q\in [q_1,q_2]\cap \frac{N}{N^2}$. 
We recall that, by definition,  $\kappa(h)=\tilde\Pb_h(X_{[1,\infty]}>0)$ (see \eqref{defKB}). At this stage, we need to prove the convergence of $\tilde\Pb_h(X_{[1,a_N]}>0)$ towards $\kappa(h)$ uniformly in $h$
in any compact subset of $(0,\beta/2)$. This is the object of the following lemma.
\begin{lemma}\label{convkappa}
For $\beta>0$, the function $x\in (0,\beta/2)\mapsto \kappa(x)$ is continuous and for $[x_1,x_2]\in (0,\beta/2)$ 
\begin{equation}\label{convunifvk}
\tilde\Pb_x(V_{[1,k]}>0)=\kappa(x)+o(1), 
\end{equation}
 where $o(1)$ is a function of $x$ and $k$ that converges to $0$ as $k\to \infty $
uniformly in $x\in [x_1,x_2]$.
Moreover, for $\gb>0$ and $x\in[0,\gb/2)$, one has
\begin{equation}\label{explicitkappa}
\kappa(x) \,=\, \frac{e^{2x}-1}{e^{x+\gb/2}-1}\,.
\end{equation}
\end{lemma}
\begin{proof}
We pick $x\in [x_1,x_2]$, $k\geq 1$  and we set $\gep:=x_1/2$, then
\begin{equation}
0\leq \tilde\Pb_x(X_{[1,k]}>0)-\kappa(x)\leq \sum_{j=k+1}^\infty \tilde\Pb_x(V_j\leq 0)\leq  \sum_{j=k+1}^\infty e^{-(\cL(x)-\cL(x-\gep))j}
\end{equation}
were we have used a Markov exponential inequality to obtain the last inequality. The fact that $\cL$ is convex and non decreasing on $[0,\beta/2)$ allows us to claim that  
for every $x\in [x_1,x_2]$ we have  $\cL(x)-\cL(x-\gep)\geq \tilde r:= \cL'(x_1-\gep) \gep =\cL'(\tfrac{x_1}{2}) \frac{x_1}{2}>0$ and therefore 
 \begin{equation}\label{unifrp}
0\leq \tilde\Pb_x(X_{[1,k]}>0)-\kappa(x)\leq \sum_{j=k+1}^\infty e^{-\tilde r j}.
\end{equation}
At this stage,  \eqref{convunifvk} is a straightforward consequence of  \eqref{unifrp}. The continuity of 
$x\mapsto \kappa(x)$ on $(0,\beta/2)$ is a consequence of the fact that $x\mapsto \tilde\Pb_x(X_{[1,k]}>0)$
is continuous on $(0,\beta/2)$ for every $k\in \N$ and of \eqref{convunifvk} which guarantees us that 
the latter sequence of functions converges uniformly to $x\mapsto \kappa(x)$ on every compact subset of 
$(0,\beta/2)$.
%
%

Let us now prove \eqref{explicitkappa}. Pick $x\in [0,\beta/2)$ and  define $\rho=\inf\{i\geq1, X_i\leq 0\}$, hence
\[1-\kappa(x)\,=\, \tilde\Pb_x(\rho<\infty)\,=\, \Ebb\big[e^{xX_\rho-\cL(x)\rho}1_{\{\rho<\infty\}}\big]\,=\,\Ebb\big[e^{xX_\rho-\cL(x)\rho}\big]\,,\]
where we used \eqref{def:Pgd} and that $\rho<\infty$ $\Pbb$-a.s.. As claimed in the proof of Lemma~\ref{exprgene}, $\rho$ and $X_\rho$ are independent, and $-X_\rho$ follows a geometric law on $\N\cup\{0\}$ with parameter $1-e^{-\beta/2}$. Moreover, $(e^{-xX_n-\cL(x)n})_{n\geq1}$ is a martingale under $\Pbb$ (recall that $\cL$ is symmetric), and it is uniformly integrable when stopped at time $\rho$; thus a stopping time argument yields that
\[\Ebb\big[e^{-\cL(x)\rho}\big]\,=\, \Ebb\big[e^{-xX_\rho}\big]^{-1}\,,\]
(we do not write the details here). Thereby,
\[1-\kappa(x)\,=\,\tilde\Pb_x(\rho<\infty)\,=\, \Ebb\big[e^{xX_\rho}\big] \Ebb\big[e^{-xX_\rho}\big]^{-1} \,=\, \frac{1-e^{x-\gb/2}}{1-e^{-x-\gb/2}}\,,\]
which finally yields \eqref{explicitkappa}.

\end{proof}
We recall that $h=\tilde h^q/2$. Lemma \ref{diffeotildeG} guarantees us that for every $q\in [q_1,q_2]$
we have $\tilde h^q\in [\tilde h^{q_1},\tilde h^{q_2}]\subset (0,\beta)$. Therefore, by using \eqref{hatanq} combined with \eqref{mnpc} and by applying Lemma \ref{convkappa} with $[x_1,x_2]=\big[\frac{\tilde h^{q_1}}{2}, \frac{\tilde h^{q_2}}{2}\big]$ we complete the proof of Lemma \ref{contromain}.

\section{Proof of Theorem \ref{thm:C}}\label{unibead}
We recall (\ref{def:omegalc}--\ref{defle}) and Remark \ref{beadtrajnco}. For $A\subset \Omega_L$ we will denote by $Z_{L,\beta}(A)$ the partition function restricted to those trajectories in 
$A$, i.e.,
$$Z_{L,\beta}(A)=\sum_{\ell\in A}  \,  e^{H_{L,\beta}(\ell)}. $$

For $L\in \N$, the function $k\mapsto  \PbLb\big( | \Imax(\ell)|\geq L-k)$ is non decreasing. Therefore, proving 
\eqref{eq:C} with $\PbLb\big( | \Imax(\ell)|\geq L-3k)$ will be sufficient.
Pick $\ell \in \Omega_L$, and let $\cJ_{k}:=\max\{j\geq 0\colon \mathfrak{X}_j\leq k\}$.
Note that, if $\ell \in \Omega_L$ has no bead starting in $\{k,\dots,L-k\}$ and if the last bead of $\ell$ starting before $k$ begins with at most $k-1$ horizontal steps, then the longest sequence of 
non-zero vertical stretches of alternating signs in $\ell$ has a total length at least $L-3k$, i.e., 
 \begin{equation}\label{twosubin}
A_{L,k}\cap B_{L,k}\subset \{| \Imax(\ell)|\geq L-3k\}.
\end{equation}
with 
\begin{align}
A_{L,k}&:=\{\ell\in \Omega_L\colon\, \mathfrak{X}\cap\{k,\dots,L-k\}=\emptyset\}\\
\nonumber B_{L,k}&:=\{\ell\in \Omega_L\colon\, \exists j\in \{1,\dots,k\}\colon \ell_{\tau_{\cJ_k}+j}\neq 0\}.
\end{align}
Since $k\mapsto  \PbLb\big( | \Imax(\ell)|\geq L-3k)$ is non decreasing and since 
\eqref{twosubin} yields 
$$\{| \Imax(\ell)|\geq L-3k\}^c\subset (A_{L,k})^{c}\, \cup \big( A_{L,k} \cap (B_{L,k})^c),$$ 
Theorem \ref{thm:C} will be proven once we show that for every $\gep>0$ there exists a $k_\gep\in \N$ and a $L_\gep\in \N$ 
such that on one hand $L_\gep\geq 3k_\gep$ and  on the other hand, for $L\geq L_\gep$
\begin{align}\label{twoineq}
\PbLb\big( (A_{L,k_\gep})^c)\leq \gep\quad \text{and} \quad \PbLb\big( A_{L,k_\gep}\cap (B_{L,k_\gep})^c)\leq \gep.
\end{align}

For simplicity we will write $\alpha=-\tilde \cG(a_\beta)$. A straightforward consequence of Theorem \ref{Th:4.1} and Corollary \ref{prop:4.2} is that there exists $0<C_1<C_2<\infty$ such that
for every $L\in \N$, 
\begin{align}\label{unifbz}
\frac{C_1 }{L^{3/4}}\, e^{\gb L -\alpha \sqrt{L}}&\leq  \hat Z^{\, \circ}_{L,\gb}\leq  Z_{L,\gb}\leq 
\frac{C_2}{L^{3/4}}\, e^{\gb L -\alpha \sqrt{L}}.
\end{align}
Thus for $k,L \in \N$ such that $L\geq 3k$ we can write
\begin{align}\label{uniper}
\nonumber  \PbLb\big((A_{L,k})^c\big)&\leq \sum_{j=k}^{L-k}  \PbLb\big(j\in \mathfrak{X}\big)= \frac{1}{Z_{L,\beta}}\sum_{j=k}^{L-k}  Z_{j,\beta}^{\,c}\  Z_{L-j,\beta}(\ell_1\geq 0)\\
\nonumber &\leq \frac{1}{Z_{L,\beta}}\sum_{j=k}^{L-k}  Z_{j,\beta}\  Z_{L-j,\beta}\\
\nonumber  &\leq \cst \, e^{\alpha \sqrt{L}} \ L^{3/4} \sum_{j=k}^{L-k}\,  \frac{1}{j^{3/4}}\,  e^{-\alpha \sqrt{j}}\ \frac{1}{(L-j)^{3/4}} \, e^{-\alpha \sqrt{L-j}}\\
\nonumber   &\leq \cst \,  \sum_{j=k}^{L-k} \frac{L^{3/4}}{j^{3/4} \, (L-j)^{3/4}}  \ e^{-\alpha (\sqrt{j}+ \sqrt{L-j}-\sqrt{L})}\\
  &\leq \cst \,  \sum_{j=k}^{L/2} \frac{1}{j^{3/4} }  \ e^{-\frac{\alpha \sqrt{j}}{2}}  \leq  \cst \,  \sum_{j=k}^{\infty} \frac{1}{j^{3/4} }  \ e^{-\frac{\alpha \sqrt{j}}{2}},
 \end{align}
 where in the second line we have used \eqref{unifbz} and in the last line we have used the convex inequality: $\sqrt{L}-\sqrt{L-j}\leq \frac{1}{2} \sqrt{j}$ for $0\leq j \leq L/2$. Since the r.h.s. in  \eqref{uniper} does not depend on $L$ and vanishes as 
 $k\to \infty$, the leftmost inequality in \eqref{twoineq} is proven. 
 \medskip
 
Let us now deal with the second inequality in  \eqref{twoineq}. For $L\geq 3k$, we partition the set $A_{L,k}\cap (B_{L,k})^c$ 
by recording $r$ (respectively $s$), the rightmost  (resp. leftmost) point in $\mathfrak{X}$ that is smaller than $k$ (resp. larger than $L-k$). 
Moreover, the fact that $\ell\in  (B_{L,k})^c$ implies that the bead which covers the interval $\{k,\dots,L-k\}$ begins with $k$ zero-length vertical stretches.
Thus,
\begin{align}\label{decplugrexc}
\nonumber Z_{L,\beta}(A_{L,k}\cap (B_{L,k})^c)&=\sum_{r=0}^{k-1} \sum_{s=L-k+1}^L Z_{L,\beta}(\{\mathfrak{X}\cap\{r,\dots,s\}=\{r,s\}\}\cap (B_{L,k})^c)\\
\nonumber  &= \sum_{r=0}^{k-1} \sum_{s=L-k+1}^L Z_{r,\beta}^{\,c}\  \hat Z^{\, \circ}_{s-r,\gb}(\ell_1=\dots=\ell_k=0)\  Z_{L-s,\beta}(\ell_1\geq 0)\\
&\leq 2\,  \sum_{r=0}^{k-1} \sum_{s=L-k+1}^L Z_{r,\beta}^{\,c}\  \hat Z^{\, \circ}_{s-r-k,\gb}\  Z_{L-s,\beta}(\ell_1\geq 0)\\
\nonumber &=2\  Z_{L-k,\beta}(\{\mathfrak{X}\cap\{r,\dots,s-k\}=\{r,s-k\}\})
\end{align}
where the third line in \eqref{decplugrexc} is obtained by observing that 
$\hat Z^{\, \circ}_{j,\gb}(\ell_1=\dots=\ell_k=0)\leq 2  \hat Z^{\, \circ}_{j-k,\gb}$ for $j\geq k+2$. The factor $2$ in the r.h.s. of the latter inequality comes from the fact that there is no constraint on the sign of the $(k+1)$-th stretch of  $\ell \in \hat\Omega_j^{\, \circ}$ satisfying $\ell_1=\dots=\ell_k=0$. 
 The r.h.s. in the fourth line of  \eqref{decplugrexc} is obviously bounded above by $2 Z_{L-k,\beta}$ and therefore,
\begin{align}\label{borner}
\PbLb\big( A_{L,k}\cap (B_{L,k})^c)&=\frac{Z_{L,\beta}(A_{L,k}\cap  (B_{L,k})^c)}{Z_{L,\beta}}\leq  \frac{2\, Z_{L-k,\beta}}{Z_{L,\beta}}.
\end{align}
It remains to use \eqref{unifbz} so that \eqref{borner} becomes 
\begin{align}\label{borner2}
\nonumber \PbLb\big( A_{L,k}\cap (B_{L,k})^c)&\leq \cst \frac{L^{3/4}}{(L-k)^{3/4}} e^{-\beta k}  e^{-\alpha (\sqrt{L-k}-\sqrt{L})} \\
&\leq \cst e^{-\beta k}  e^{-\alpha (\sqrt{L-k}-\sqrt{L})}.
\end{align}
For $\gep >0$ we pick $k_\gep\in \N$ such that $\cst^{-\beta k_\gep}\leq \gep/2$. Moreover, a straightforward computation gives us that 
$\lim_{L\to \infty} \sqrt{L-k_\gep}-\sqrt{L}=0$. This completes the proof of the rightmost inequality in \eqref{twoineq} and 
ends the proof of Theorem \ref{thm:C}.

\appendix

\section{Local limit estimates, proof of Lemma \ref{estimG}}\label{pr:estimG}
We will prove Lemma \ref{estimG} subject to Proposition \ref{llta} and Lemma \ref{lem62} that are stated below and which were proven  in \cite[Section 6]{CNPT18}. To that aim, we recall 
 (\ref{def:cDgb}--\ref{hrt}) and we set 
 \begin{equation}\label{defff}
 {\bf{B}}(\bh)=\text{Hess}\  \cL_{\Lambda}(\bh), \quad \bh\in \cD,
 \end{equation}
 and 
 \begin{equation}\label{defcL}
 f_{\bh}(\bar x) =\frac{(\text{Det}({\bf{B}}(\bh)))^{-1/2}}{2\pi}\, \exp\big(-\tfrac12 \langle {\bf{B}}(\bh)^{-1} \bar x, \bar x\rangle\big), \quad \bar x\in \mathbb{R}^2.
 \end{equation}
  \smallskip
 
\begin{proposition}{\cite[Proposition 6.1]{CNPT18}}\label{llta}
For $[q_1,q_2]\subset (0,\infty)$, we have
$$\sup_{q\in [q_1,q_2]\cap \frac{\N}{n^2}}\ \sup_{y,z\in \mathbb{Z}} \big |n^2\, \bP_{n,\, h_n^q}\big(A_{n-1}=n^2 q+y, X_n=z)-f_{\tilde \bh(q,0)}\big(\tfrac{y}{n^{3/2}},\tfrac{z}{\sqrt{n}}\big)\big|\underset{n\to \infty}{\longrightarrow} 0.$$
\end{proposition}
 \smallskip

\begin{lemma}{\cite[Lemma 6.2]{CNPT18}}\label{lem62}
For $[q_1,q_2]\subset (0,\infty)$, there exist three positive constants $C', C_1, \lambda$ such that for $N$ large enough, the following bound holds true
$$\bE_{N,h_N^q}\big[e^{-\lambda X_j}\big]\leq C'\, e^{-C_1\, j}, \quad \text{for}\quad   j\leq \tfrac{N}{2} \ \text{and} \ \  q\in [q_1,q_2]\cap \tfrac{\N}{N^2}.$$ 
\end{lemma}
\medskip

We recall \eqref{defG} and we relax the constraint $\{X_{[0,N_2]}>-x_1\}$ in $G_{N,\bar x,\bar a}^{\,q}$ to define $\bar G_{N,\bar x,\bar a}^{\,q}$, i.e.,
\begin{equation}\label{defGbar}
\bar G_{N,\bar x,\bar a}^{\,q}:=\Pb_{N_2, h_{N_2}^{q}}\Big[\Lambda_{N_2}=\big(\tfrac{qN^2-a_1-a_2}{N_2}-x_1\big(1-\tfrac1{N_2}\big),\, x_2-x_1\big)\Big].
\end{equation}
\smallskip

\noindent Remark \ref{rmsc} guarantees us
that there exist a $K>0$ and a $N_0\in \N$ such that  $h_{N}^q\in [K,\beta-K]$ for $N\geq N_0$ and $q\in [q_1,q_2]\cap \frac{\N}{N^2}$.
Therefore, there exists a $c_1>0$ such that for $N$ large enough,  $q\in [q_1,q_2]\cap \frac{\mathbb{N}}{N^2}$ and $x_1\in \cC_N$ we have
\begin{equation}\label{bottle}
x_1 \geq \bE_{N,h_N^q}(X_{a_N})-(a_N)^{3/4} \geq  a_N\,  \tilde\bE_{\frac{K}{2}}(X_1)-(a_N)^{3/4} \geq c_1 a_N\,.
\end{equation}
As a consequence,  we can write 
\begin{equation}\label{encadr}
\bar G_{N,\bar x,\bar a}^{\,q} \geq G_{N,\bar x,\bar a}^{\,q}\geq \bar G_{N,\bar x,\bar a}^{\,q}-\hat G_{N,\bar x}^{\,q}
\end{equation}
with
\begin{align}\label{defhatg}
\hat G_{N,\bar x}^{\,q} &=\sum_{i=1}^{N_2-1} \bP_{N_2,  h_{N_2}^q}\big(X_{N_2}=x_2-x_1, X_i\leq -c_1 a_N\big).
\end{align}

Until the end of the present section, every function $o(1)$ converges to $0$ as $N\to \infty$ uniformly in $q\in  [q_1,q_2]\cap \frac{\mathbb{N}}{N^2}$ and $(\bar x,\bar a)\in \cH_N$ (we will omit to repeat this uniformity to lighten notations).  Thanks to  \eqref{encadr}, the proof of Lemma \ref{estimG} will be complete once we show that 
\begin{equation}\label{convuniftp}
N^2\, \bar G_{N,\bar x,\bar a}^{\,q}=\frac{\big(\vartheta(\tilde h^q)\big)^{-\frac{1}{2}}}{2\pi} +o(1),
\end{equation}
and that 
\begin{equation}\label{convunifts}
N^2\, \hat G_{N,\bar x}^{\,q} =o(1).
\end{equation}
\medskip

Let us start with \eqref{convuniftp} by stating a remark about the Hessian matrices $\bf{B}$ (see \eqref{defff}).

\begin{remark}\label{remhess}
From \eqref{def:cL} we deduce easily that $\cL''(h)>0$ for $h\in (-\beta/2,\beta/2)$.  As a consequence, we can rule out the equality case when applying the Cauchy-Schwartz inequality to \eqref{defvartheta}. Therefore, 
 \begin{equation}\label{detb}
 \vartheta(\tilde h^q)=\text{Det}({\bf{B}}(\tilde \bh(q,0)))>0,\quad q\in \R.
 \end{equation}
Since $\cL_\Lambda$ is convex on $\cD$, \eqref{detb} is sufficient to assert that  ${\bf{B}}(\tilde  \bh(q,0))$ is a symmetric positive-definite matrix for every $q\in \R$. Thus, the eigenvalues of  ${\bf{B}}(\tilde  \bh(q,0))$ are positive and continuous in $q\in \R$.  This yields that, for 
 $[q_1,q_2]\subset (0,\infty)$, the eigenvalues of ${\bf{B}}(\tilde  \bh(q,0))$ are bounded above and below by  positive constants that are uniform in $q\in [q_1,q_2]$.  Therefore, 
 there exists  a compact subset  $\cK\subset(0,\infty)$ such that $\vartheta(\tilde h^q)\in \cK$ for every 
 $q\in [q_1,q_2]$. From \eqref{defcL}, the latter implies that $f_{\tilde \bh(q,0)}$ is bounded from above on $\R^2$ , uniformly in  $q\in [q_1,q_2]$.
\end{remark}

We resume the proof of \eqref{convuniftp}. By definition of $N_2=N-2a_N$ we easily show that 
\begin{equation}\label{premco}
N^2\, \bar G_{N,\bar x,\bar a}^{\,q}=(N_2)^2\, \bar G_{N,\bar x,\bar a}^{\,q}\,  (1+o(1)).
\end{equation}
Thus, we apply Proposition \ref{llta}  with $n=N_2$,  $y=q(N^2-(N_2)^2)-a_1-a_2-x_1 \frac{N_2-1}{N_2}$ and $z=x_2-x_1$, and we obtain that 
\begin{equation}\label{secco}
(N_2)^2\, \bar G_{N,\bar x,\bar a}^{\,q}= f_{\tilde \bh(q,0)}\Big(\frac{y}{(N_2)^{3/2}}, \frac{z}{\sqrt{N_2}}\Big)+  o(1).
\end{equation}
In Section \ref{Step4},  to derive \eqref{bdB2} and \eqref{perturbq}, we have proven that there exists a $C>0$ (uniform in $q\in [q_1,q_2]$) such that $(\bar x,\bar a)\in \cH_N$
implies $|x_1|+|x_2|\leq C a_N$ and $|a_1|+|a_2|\leq C a_N^2$. Moreover, 
provided $C$ is chosen large enough we also have  $N^2-(N_2)^2\leq C\, N\,a_N$.  Therefore, provided $C$ is chosen large enough,
\begin{equation}\label{borneyz}
\frac{|y|}{(N_2)^{3/2}}\leq C \frac{a_N}{\sqrt{N}}\quad \text{and} \quad \frac{|z|}{\sqrt{N_2}}\leq C \frac{a_N}{\sqrt{N}},\quad \text{for}\    (\bar x,\bar a)\in \cH_N, \ 
q\in [q_1,q_2].
 \end{equation}
By remark \ref{remhess}, we know that the eigenvalues of ${\bf B}(\tilde \bh(q,0))^{-1}$ are bounded above uniformly in $q\in[q_1,q_2]$, so that 
\eqref{borneyz} allows us to write
\begin{equation}\label{trco}
\sup_{q\in [q_1,q_2]}\ \sup_{( \bar x,\bar a)\in \cH_N} \Big|f_{\tilde \bh(q,0)}\big(\tfrac{y}{(N_2)^{3/2}},\tfrac{z}{\sqrt{N_2}}\big)-f_{\tilde \bh(q,0)}(0,0)\Big|\underset{N\to \infty}{\longrightarrow} 0.
\end{equation}
At this stage, we complete the proof of \eqref{convuniftp} by observing that  $f_{\tilde \bh(q,0)}(0,0)=\vartheta(\tilde h^q)^{-1/2}\, \frac{1}{2\pi}$ and by combining \eqref{premco}, \eqref{secco} and \eqref{trco} with the fact that  $f_{\tilde \bh(q,0)}$ is bounded from above on $\R^2$, uniformly in  $q\in [q_1,q_2]$.
\bigskip

It remains to prove \eqref{convunifts}. 
We apply \eqref{timereveq} with $n=N_2$ and $h= h_{N_2}^q$ to claim that 
for $i\in \{1,\dots,N_2-1\}$,
\begin{align}\label{trev}
\bP_{N_2, h_{N_2}^q}(X_{N_2}=x_2-x_1,&\,  X_i\leq -c_1 a_N)\\
\nonumber &=\bP_{N_2,\,  h_{N_2}^q}(X_{N_2}=x_1-x_2, X_{N_2-i}\leq -c_1 a_N+x_1-x_2).
\end{align}
Since $|x_1-x_2| \leq 2 (a_N)^{3/4}$ we have that, for $N$ large enough, $-c_1 a_N+x_1-x_2\leq -c_1 a_N/2$. Therefore, by  using \eqref{trev}
with $i\in \{\frac{N_2}{2},\dots,N_2-1\}$ we obtain that 
\begin{align}\label{defhatg2}
\bP_{N_2,\,  h_{N_2}^q}\big(X_{N_2}=x_2-x_1, X_i\leq -c_1 a_N\big)\leq  \bP_{N_2,\,  h_{N_2}^q}\big(X_{N_2-i}\leq -\tfrac{c_1}{2} a_N\big).
\end{align}
Coming back to \eqref{defhatg} and using \eqref{defhatg2} for $i\geq N_2/2$ we can bound $\hat G_{N,\bar x}^{\,q}$ from above as
\begin{align}\label{boundhatg3}
\nonumber\hat G_{N,\bar x}^{\,q} &\leq\sum_{i=1}^{N_2/2} \bP_{N_2, h_{N_2}^q}\big(X_{N_2}=x_2-x_1, X_i\leq -c_1 a_N\big)+\sum_{N_2/2}^{N_2-1} \bP_{N_2, h_{N_2}^q}\big(X_{N-i}\leq -\tfrac{c_1}{2} a_N\big)\\
&\leq 2 \sum_{i=1}^{N_2/2} \bP_{N_2, h_{N_2}^q}\big( X_i\leq -\tfrac{c_1}{2} a_N\big)\leq 2  \sum_{i=1}^{N_2/2} 
\bE_{N_2, h_{N_2}^q}\big(e^{-\lambda X_i}\big) \ \ e^{-\tfrac{\lambda\, c_1\, a_N}{2} }\\
&\leq 2 \ e^{-\tfrac{\lambda\, c_1\, (\log N)^2}{2} }  \sum_{i=1}^{\infty} 
C'\, e^{-C_1\, j}=o(1) \, \tfrac{1}{N^2},
\end{align}
where we have used a Chernov inequality to obtain the second inequality in the second line, and Lemma~\ref{lem62} for the last line. This completes the proof of Lemma \eqref{estimG}.

\section{Functional estimates, proof of Proposition \ref{controlderiv}}\label{pr:propderiv}

Let us start with the case $j=0$. We recall \eqref{defGn} and we set 
\begin{equation}\label{deffN}
f_{N,h}(x):=\cL(\tfrac{h}{2}(1+\tfrac{1}{N})-h\, \tfrac{x}{N})
\end{equation}
and therefore $N \cG_N(h)=\sum_{i=1}^{N} f_{N,h}(i)$. We apply the Euler-Maclaurin summation formula (see e.g.  \cite[Theorem 0.7]{GT15}
) and since $ f_{N,h}$ is $\cC^2$  we obtain 
that 
\begin{align}\label{appEuMa}
 N\cG_n(h)=&A(N,h)+B(N,h)
\end{align}
with 
\begin{equation}\label{defANh}
A(N,h):=\frac{f_{N,h}(1)+f_{N,h}(N)}{2}+\int_1^N f_{N,h}(x)\,  dx
\end{equation}
and 
\begin{equation}\label{defB}
B(N,h):=\frac{1}{2}\,  \int_1^N f''_{N,h}(x) \, \big(B_2(0)-B_2(x-\lfloor x\rfloor)\big)\, dx.
\end{equation}
where $B_2$ is seconde Bernoulli polynomial.

We start by considering $A(N,h)$. Recalling the definition of $\cL$ in \eqref{def:cL} and the fact that $\Pbb$ is symmetric, we claim that $\cL$ is even and therefore 
\begin{equation}\label{defcA1}
\frac{f_{N,h}(1)+f_{N,h}(N)}{2}=\cL\Big(\frac{h}{2}-\frac{h}{2N}\Big).
\end{equation}
We recall the definition of $\cG$ in \eqref{deftiG} and a straightforward computation gives 
\begin{align}\label{defcA2}
\int_1^N f_{N,h}(x)\,  dx &= N\int_{1/2N}^{1-1/2N} \cL(h(\tfrac12-y)) \,  dy\\
\nonumber &=N \cG(h) -2N \int_0^{1/2N} \cL(h(\tfrac12-z)) \, dz
\end{align}
where we have used the change of variable $y=x/N$ to get the first equality
and the parity of $\cL$ combined with the change of variable $z=1-y$ to obtain the second equality. Obviously, for $N$ large enough and for every  $h\in [-\beta+K,\beta-K]$ we have that both  $\frac{h}{2}$ and  $\frac{h}{2}(1-\frac{1}{N})$ belong to $\mathcal{R}_K:=\big[-\frac{\beta}{2}+\frac{K}{2},\frac{\beta}{2}-\frac{K}{2}\big]$.
Since $\cL$ is $\cC^1$  we  set   $C'_K:=\max\{|\cL'(x)|, x\in \mathcal{R}_K\}$ and for $N$ large enough
\begin{equation}
\Big|\cL\Big(\frac{h}{2}-\frac{h}{2N}\Big)-\cL\Big(\frac{h}{2}\Big)\Big|\leq C'_K \frac{|h|}{2N}\leq  \beta \, C'_K\,\frac{1}{2N}
\end{equation}
and 
\begin{equation}\label{premine}
\Big|2N \int_0^{1/2N} \cL(h(\tfrac12-z))dz -\cL(\tfrac{h}{2})\Big|\leq 2N \int_0^{1/2N}  C'_K \, |h| \, z dz\leq \beta\,  C'_K\,   \frac{1}{4N}.
\end{equation}
Combining (\ref{defcA1}--\ref{premine}) we claim that, for $N$ large enough
\begin{equation}\label{Aenc}
|A(N,h)-N\, \cG(h)|\leq \beta\,  C'_K\,  \tfrac1N, \quad \forall  h\in [-\beta+K,\beta-K].
\end{equation}
It remains to consider $B(N,h)$ (recall \eqref{defB}). To that aim, we compute the second derivative of $f''_{N,h}$ and obtain
\be\label{fderivative}
f''_{N,h}(x)=\tfrac{h^2}{N^2}\, \cL''\big[\tfrac{h}{2}(1+\tfrac1N)-h\tfrac{x}{N}\big], \quad x\in [1,N].
\ee
It turns out that for $x\in [1,N]$ we have 
$|\tfrac{h}{2}(1+\tfrac1N)-h\tfrac{x}{N}|\leq \frac{|h|}{2}$. Thus, $h\in [-\beta+K,\beta-K]$ yields $\tfrac{h}{2}(1+\tfrac1N)-h\tfrac{x}{N}\in \mathcal{R}_K$
and since $\cL$ is $\cC^2$  we  set  $C''_K:=\max\{|\cL''(x)|, x\in \mathcal{R}_K\}$ and we obtain 
$|f''_{N,h}(x)|\leq C''_K \beta^2/N^2$. As a consequence, we can use \eqref{defB} to write 
\begin{align}\label{defB3}
\big|B(N,h)\big|&\leq C''_K \frac{\beta^2}{N}\, \max\{|B_2(u)|, u\in [0,1]\}.
\end{align}
At this stage, it remains to combine \eqref{appEuMa}, \eqref{Aenc} and \eqref{defB3} to complete the proof of Proposition  \ref{controlderiv} in the case $j=0$. 

The proof for the case $j=1$ being very close in spirit to that for $j=0$, we will give less details in this case. 
We can repeat (\ref{deffN}--\ref{defB}) with $N \cG'_N(h)$ instead of $N \cG_N(h)$ and  after redefining  $f_{N,h}$ as
\begin{equation}\label{deffN2}
f_{N,h}(x):=\big(\tfrac12(1+\tfrac{1}{N})-\tfrac{x}{N}\big)\,  \cL'[ \tfrac{h}{2}(1+\tfrac{1}{N})-h\, \tfrac{x}{N} ].
\end{equation}
Using that $\cL'$ is odd, introducing the function  $g_h(u):=(\frac12-u)\cL'[h (\frac12-u)]$ and computing $\cG'$ from \eqref{deftiG},   we obtain that in this case
\begin{align}\label{defA2}
\nonumber A(N,h)&=N \int_{1/2N}^{1-1/2N} (\tfrac12-y)\,  \cL'[h(\tfrac12-y)]\,  dy+\tfrac12\,  (1-\tfrac{1}{N})\,  \cL'[h(\tfrac12-\tfrac{1}{2N})]\\
&=N \cG'(h)+ g(\tfrac{1}{2N})-2N \int_0^{1/2N} g(u) \, du.
\end{align}
Taking the derivative of $g$ we obtain $g'(u)=-\cL'[h (\frac12-u)] - h(\frac12-u)\cL''[h (\frac12-u)]$ so that 
\begin{equation}\label{bsupderg}
\sup_{u\in [0,1/2N]} \sup_{h\in [-\beta+K,\beta-K]} |g'(u)|\,\leq\, C'_K+\gb C''_K
\end{equation}
Therefore,
\begin{align}\label{boundso2te}
\nonumber \Big|g(\tfrac{1}{2N})-2N \int_0^{1/2N} g(u) \, du \Big| \,&\leq\, \big| g(\tfrac{1}{2N})-g(0) \big |+2N  \int_0^{1/2N} |g(u)-g(0)| \, du\\
&\leq\, (C'_K+\gb C''_K)\frac1N.
\end{align}
It remains to consider $B(N,h)$ for which we need to compute $f''_{N,h}$, i.e., for $x\in [1,N]$,
\be\label{compfsec}
f''_{N,h}(x)= \tfrac{2 h}{N^2}\, \cL''\big[\tfrac{h}{2}(1+\tfrac1N)-h\tfrac{x}{N}\big]+  \tfrac{ h^2}{N^2}\,\big( \tfrac{1}{2}(1+\tfrac1N)-\tfrac{x}{N}\big) \cL'''\big[\tfrac{h}{2}(1+\tfrac1N)-h\tfrac{x}{N}\big].
\ee
Thus, after defining  $C'''_K:=\max\{|\cL'''(x)|, x\in \mathcal{R}_K\}$ and by mimicking the former proof we obtain $|f'''_{N,h}(x)|\leq (2\, C''_K\,  \beta+\beta^2\, C'''_K)/N^2$ for $x\in [1,N]$. 
This is sufficient to claim (from \eqref{defB}) that 
\begin{align}\label{defB4}
\big|B(N,h)\big|&\leq (2\, C''_K\,  \beta+\beta^2\, C'''_K)\,  \frac{1}{N}\, \max\{|B_2(u)|, u\in [0,1]\}.
\end{align}
We combine \eqref{defA2}, \eqref{boundso2te} and \eqref{defB4} and it completes the proof of Proposition  \ref{controlderiv} in the case $j=1$.

\bibliographystyle{plain}
\bibliography{biblio.bib}

\end{document}